\documentclass[12pt]{article}
\usepackage{graphicx}
\usepackage{amsthm}
\usepackage{amssymb}
\usepackage{amsmath}
\usepackage{hyperref}
\usepackage{float}
\usepackage{graphicx}
\usepackage{subcaption}
\usepackage{ mathrsfs }
\usepackage[font=small,labelfont=bf]{caption}
\graphicspath{ {./images/} }
\usepackage{tikz}

\usepackage[left=30mm,right=30mm,top=20mm,bottom=20mm]{geometry}

\usepackage[maxbibnames=99]{biblatex}
\usepackage{comment}

\theoremstyle{plain}
\newtheorem{theorem}{Theorem}[section]
\newtheorem{corollary}[theorem]{Corollary}
\newtheorem{lemma}[theorem]{Lemma}

\newtheorem{proposition}[theorem]{Proposition}

\theoremstyle{definition}

\newtheorem{example}[theorem]{Example}
\theoremstyle{remark}

\title{Propagation processes on (hyper)graphs:\\ where zero forcing and burning meet}

\author{Aida Abiad\thanks{Corresponding author,  \texttt{a.abiad.monge@tue.nl}, Department of Mathematics and Computer Science, Eindhoven University of Technology, Eindhoven, The Netherlands\\Department of Mathematics and Data Science of Vrije Universiteit Brussel, Brussels, Belgium}\qquad 
\qquad Pax Mallee\thanks{\texttt{p.mallee@student.tue.nl},  Department of Mathematics and Computer Science, Eindhoven University of Technology, Eindhoven, The Netherlands}
}

\date{}

\begin{document}

\maketitle

\begin{abstract}
The burning and forcing processes are both instances of propagation processes on graphs that are commonly used to model real-world spreading phenomena. The contribution of this paper is two-fold. We first establish a connection between these two propagation processes via hypergraphs. We do so by showing a sharp upper bound on the zero forcing number of the incidence graph of a hypergraph in terms of the lazy burning number of the hypergraph, which builds up on and improves a result by Bonato, Jones, Marbach, Mishura and Zhang (\emph{Theor. Comput. Sci.}, 2025). Secondly, we deepen the understanding of the role of the burning process in the context of graph spectral characterizations, whose goal is to understand which graph properties are encoded in the spectrum. While for several graph properties, including the zero forcing number, it is known that the spectrum does not encode them, this question remained open for the burning number. We solve this problem by constructing infinitely many pairs of cospectral graphs which have a different burning number.
\end{abstract}

\section{Introduction}

The burning and forcing processes often appear across mathematics and computer science, but also in other fields such as statistical mechanics \cite{CLR1979}, physics \cite{ARENAS200893} or social network analysis \cite{G1978}), among others. In such fields, burning and zero forcing are used to model technical or societal processes. For an overview of the different models and
applications, refer to the survey \cite{survey-burning-number} and the book \cite{hogben2022inverse}.

The concept of graph burning, which was introduced in \cite{originburning}, is a deterministic, discrete-time model for the spread of social contagion on a graph. A social network can be modeled by a graph in which the vertices represent people and the edges represent relationships. For example, on a graph whose vertices correspond to Facebook users, edges may represent users who are Facebook friends. Information, such as gossip, rumors, or memes, can spread from vertex to vertex over time along edges of the graph and we model the process using discrete time-steps called rounds. Such information may not stem from a single source vertex; there may be a number of sources that appear over time in the graph. Authors often refer to vertices as unburned (unaware of the information) and burned (aware of the information) since rumors can appear to spread swiftly like fire. We will instead use colors to depict these states: uncolored for unburned and red for burned. Although computing the burning number is NP-complete (it is so even when restricted
to trees of maximum degree three; see \cite[Theorem 6]{bessy2017}), several authors have established the burning number for several graph classes, see e.g. \cite{A1992,TN2025}. Hypergraph burning was introduced in \cite{Burgess2024ExtendingGB} as a natural extension of graph burning. The rules for hypergraph burning are identical to those for burning graphs, except for how burning propagates within a hyperedge. Recall that in a hypergraph, a singleton edge contains exactly one vertex. In hypergraphs, the
burning spreads to a vertex $v$ in round $r$ if and only if there is a non-singleton
hyperedge $\{v, u_1,\dots, u_k\}$ such that $v$ was not burned and each of $u_1, u_2,\dots, u_k$ was burned at the end of round $r-1$. In particular, a vertex becomes burned if it is the only unburned vertex in a hyperedge. A natural variation of burning that is our principal focus here is \emph{lazy hypergraph burning}, where a set of vertices is chosen to burn in the first round only; no other vertices are chosen to burn in later rounds. The \emph{lazy burning number} of $H$, denoted $b_L(H)$ and introduced in \cite{hypergraph}, is the minimum cardinality of a set of vertices burned in the first round that eventually burns all vertices.

The lazy burning was introduced in order to establish a link with another well-known graph process: zero forcing. 
Zero forcing is a graph coloring process used to model spreading phenomena in real-world scenarios. It can also be viewed as a single-player combinatorial game on a graph, where the player's goal is to select a subset of vertices of minimum cardinality that eventually leads to all vertices of the graph being colored. In lazy burning, only one set of vertices is chosen to burn in the first round. In hypergraphs, burned vertices spread when all but one vertex in a hyperedge is burned. The lazy burning number is the minimum number of initially burned vertices that eventually burn all vertices. Recently, the authors of \cite{hypergraph} showed a characterization of lazy burning sets in hypergraphs via zero forcing sets in their incidence graphs. This was then used to derive a lower bound on the lazy burning of a hypergraph in terms of the zero forcing number of its incidence graph, see \cite[Theorem 25]{hypergraph} and the text after. However, the tightness of this bound was not discussed. Motivated by understanding the tightness of the bound derived from \cite[Theorem 25]{hypergraph}, and with the aim of further understanding the connection between the burning process and the zero forcing process, in this paper we derive an improved tight upper bound on the zero forcing number of the incidence graph of a hypergraph in terms of the lazy burning number of the hypergraph (see Theorem \ref{thm:improved bound}). We also provide some hypergraph classes that meet the new bound with equality (see Proposition \ref{propo:tightnessnewboundburningzeroforcingprocesses}).

In the second part of this paper we focus on the problem of graph spectral characterization in connection with the burning number. Spectral graph characterizations seek to study relations between the spectrum of a graph and its combinatorial properties by investigating the spectrum of the associated matrix. Some properties of a graph are known to be characterized by the spectrum of the associated matrix \cite{adjacnecy-matrix-strong-product}. However, not all graph properties can be characterized by the graph spectrum. If that is not the case, then there exists a pair of non-isomorphic graphs with the same spectrum (\emph{cospectral graphs}) that do not share that same property, instances of such graph properties were investigated, among others, in \cite{cospectral-regularity,cospectral-hamiltonian,cospectral-matchin}. Recently, also the zero forcing number has been shown not to be characterized by the spectrum \cite{aida-zero-forcing}. In the second part of this paper we continue this line of research by showing that the burning number of a graph does not follow from the spectrum of the graph adjacency matrix (see Corollary \ref{coro:burningnumbernotDS}), showing yet another analogy between both graph processes. We do so by providing a construction of two nonsiomorphic graph classes that are cospectral but have different burning numbers.

\section{Preliminaries}\label{sec:preliminaries}

Throughout this paper, we consider $G=(V(G),E(G))$ to be an undirected simple graph of order $|V|=n$. The \textit{adjacency matrix} of a graph $G=(V(G),E(G))$ is the $n\times n$-matrix $A$ where the $i,j$ entry $A_{ij}$ is $1$ if $\{v_i,v_j\}\in E(G)$ and $0$ otherwise. The \textit{spectrum} of a graph is the set of its adjacency eigenvalues.

A \textit{hypergraph} $H$ consists of a non-empty vertex set $V(H)$ vertices and a collection $E(H)$ of subsets of $V(H)$ whose elements are called \textit{hyperedges}.
An edge that contains exactly one vertex is called a \textit{singleton} hyperedge.

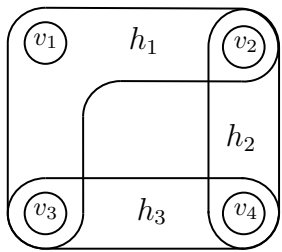
\begin{figure}[htp!]
\centering
\tikzset{every picture/.style={line width=0.75pt}} 

\begin{tikzpicture}[x=0.75pt,y=0.75pt,yscale=-1,xscale=1]

\draw   (125,123.5) .. controls (125,117.7) and (129.7,113) .. (135.5,113) .. controls (141.3,113) and (146,117.7) .. (146,123.5) .. controls (146,129.3) and (141.3,134) .. (135.5,134) .. controls (129.7,134) and (125,129.3) .. (125,123.5) -- cycle ;
\draw   (125,209.5) .. controls (125,203.7) and (129.7,199) .. (135.5,199) .. controls (141.3,199) and (146,203.7) .. (146,209.5) .. controls (146,215.3) and (141.3,220) .. (135.5,220) .. controls (129.7,220) and (125,215.3) .. (125,209.5) -- cycle ;
\draw   (225,209.5) .. controls (225,203.7) and (229.7,199) .. (235.5,199) .. controls (241.3,199) and (246,203.7) .. (246,209.5) .. controls (246,215.3) and (241.3,220) .. (235.5,220) .. controls (229.7,220) and (225,215.3) .. (225,209.5) -- cycle ;
\draw   (225,125.5) .. controls (225,119.7) and (229.7,115) .. (235.5,115) .. controls (241.3,115) and (246,119.7) .. (246,125.5) .. controls (246,131.3) and (241.3,136) .. (235.5,136) .. controls (229.7,136) and (225,131.3) .. (225,125.5) -- cycle ;
\draw    (135.5,191.7) -- (235.5,191.7) ;
\draw  [draw opacity=0] (234.27,191.7) .. controls (234.27,191.7) and (234.27,191.7) .. (234.27,191.7) .. controls (244.78,191.7) and (253.3,199.67) .. (253.3,209.5) .. controls (253.3,219.33) and (244.78,227.3) .. (234.27,227.3) -- (234.27,209.5) -- cycle ; \draw   (234.27,191.7) .. controls (234.27,191.7) and (234.27,191.7) .. (234.27,191.7) .. controls (244.78,191.7) and (253.3,199.67) .. (253.3,209.5) .. controls (253.3,219.33) and (244.78,227.3) .. (234.27,227.3) ;  
\draw    (135.5,227.3) -- (235.5,227.3) ;
\draw  [draw opacity=0] (253.3,208.27) .. controls (253.3,208.27) and (253.3,208.27) .. (253.3,208.27) .. controls (253.3,218.78) and (245.33,227.3) .. (235.5,227.3) .. controls (225.67,227.3) and (217.7,218.78) .. (217.7,208.27) -- (235.5,208.27) -- cycle ; \draw   (253.3,208.27) .. controls (253.3,208.27) and (253.3,208.27) .. (253.3,208.27) .. controls (253.3,218.78) and (245.33,227.3) .. (235.5,227.3) .. controls (225.67,227.3) and (217.7,218.78) .. (217.7,208.27) ;  
\draw    (217.7,125.5) -- (217.7,209.5) ;
\draw    (253.3,125.5) -- (253.3,209.5) ;
\draw  [draw opacity=0] (135.5,227.3) .. controls (135.5,227.3) and (135.5,227.3) .. (135.5,227.3) .. controls (124.99,227.3) and (116.48,219.33) .. (116.48,209.5) .. controls (116.48,199.67) and (124.99,191.7) .. (135.5,191.7) -- (135.5,209.5) -- cycle ; \draw   (135.5,227.3) .. controls (135.5,227.3) and (135.5,227.3) .. (135.5,227.3) .. controls (124.99,227.3) and (116.48,219.33) .. (116.48,209.5) .. controls (116.48,199.67) and (124.99,191.7) .. (135.5,191.7) ;  
\draw  [draw opacity=0] (217.7,125.5) .. controls (217.7,125.5) and (217.7,125.5) .. (217.7,125.5) .. controls (217.7,114.99) and (225.67,106.48) .. (235.5,106.48) .. controls (245.33,106.48) and (253.3,114.99) .. (253.3,125.5) -- (235.5,125.5) -- cycle ; \draw   (217.7,125.5) .. controls (217.7,125.5) and (217.7,125.5) .. (217.7,125.5) .. controls (217.7,114.99) and (225.67,106.48) .. (235.5,106.48) .. controls (245.33,106.48) and (253.3,114.99) .. (253.3,125.5) ;  
\draw  [draw opacity=0] (116.48,123.5) .. controls (116.48,113.67) and (124.99,105.7) .. (135.5,105.7) -- (135.5,123.5) -- cycle ; \draw   (116.48,123.5) .. controls (116.48,113.67) and (124.99,105.7) .. (135.5,105.7) ;  
\draw  [draw opacity=0] (154.52,209.5) .. controls (154.52,219.33) and (146.01,227.3) .. (135.5,227.3) .. controls (124.99,227.3) and (116.48,219.33) .. (116.48,209.5) -- (135.5,209.5) -- cycle ; \draw   (154.52,209.5) .. controls (154.52,219.33) and (146.01,227.3) .. (135.5,227.3) .. controls (124.99,227.3) and (116.48,219.33) .. (116.48,209.5) ;  
\draw    (116.48,123.5) -- (116.48,209.5) ;
\draw  [draw opacity=0] (234.5,106.12) .. controls (239.01,106.42) and (243.53,108.36) .. (247.09,111.91) .. controls (254.52,119.34) and (254.9,131) .. (247.95,137.95) .. controls (244.33,141.58) and (239.42,143.21) .. (234.5,142.88) -- (234.5,124.5) -- cycle ; \draw   (234.5,106.12) .. controls (239.01,106.42) and (243.53,108.36) .. (247.09,111.91) .. controls (254.52,119.34) and (254.9,131) .. (247.95,137.95) .. controls (244.33,141.58) and (239.42,143.21) .. (234.5,142.88) ;  
\draw    (135.5,105.7) -- (234.5,106.12) ;
\draw  [draw opacity=0] (154.48,160.5) .. controls (154.48,150.67) and (162.99,142.7) .. (173.5,142.7) -- (173.5,160.5) -- cycle ; \draw   (154.48,160.5) .. controls (154.48,150.67) and (162.99,142.7) .. (173.5,142.7) ;  
\draw    (154.48,160.5) -- (154.52,209.5) ;
\draw    (173.5,142.7) -- (234.5,142.88) ;

\draw (128,203) node [anchor=north west][inner sep=0.75pt]  [font=\footnotesize] [align=left] {$\displaystyle v_{3}$};
\draw (229,203) node [anchor=north west][inner sep=0.75pt]  [font=\footnotesize] [align=left] {$\displaystyle v_{4}$};
\draw (128,117) node [anchor=north west][inner sep=0.75pt]  [font=\footnotesize] [align=left] {$\displaystyle v_{1}$};
\draw (229,119) node [anchor=north west][inner sep=0.75pt]  [font=\footnotesize] [align=left] {$\displaystyle v_{2}$};
\draw (176,114) node [anchor=north west][inner sep=0.75pt]   [align=left] {$\displaystyle h_{1}$};
\draw (225,163) node [anchor=north west][inner sep=0.75pt]   [align=left] {$\displaystyle h_{2}$};
\draw (180,200) node [anchor=north west][inner sep=0.75pt]   [align=left] {$\displaystyle h_{3}$};

\end{tikzpicture}
\caption{A hypergraph with 4 vertices and 3 hyperedges.}
\label{fig:hypergraph-ex}
\end{figure}

For a hypergraph, one can define a graph that describes the incidence relationship between the vertices of the hypergraph and its hyperedges.

The \textit{incidence graph} of a hypergraph $H$, denoted $IG(H)$, is a bipartite graph with vertex set $V(H)\cup E(H)$ and edge set $\{vh\mid v\in V(H),h\in E(H) \text{ such that } v\in h\}$. 

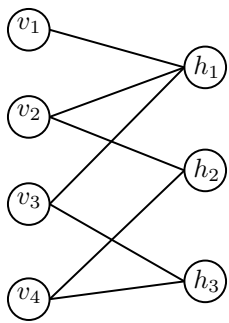
\begin{figure}[htp!]
    \centering
\tikzset{every picture/.style={line width=0.75pt}} 

\begin{tikzpicture}[x=0.75pt,y=0.75pt,yscale=-1,xscale=1]

\draw   (262,74.5) .. controls (262,68.7) and (266.7,64) .. (272.5,64) .. controls (278.3,64) and (283,68.7) .. (283,74.5) .. controls (283,80.3) and (278.3,85) .. (272.5,85) .. controls (266.7,85) and (262,80.3) .. (262,74.5) -- cycle ;
\draw   (262,118.5) .. controls (262,112.7) and (266.7,108) .. (272.5,108) .. controls (278.3,108) and (283,112.7) .. (283,118.5) .. controls (283,124.3) and (278.3,129) .. (272.5,129) .. controls (266.7,129) and (262,124.3) .. (262,118.5) -- cycle ;
\draw   (262,162.5) .. controls (262,156.7) and (266.7,152) .. (272.5,152) .. controls (278.3,152) and (283,156.7) .. (283,162.5) .. controls (283,168.3) and (278.3,173) .. (272.5,173) .. controls (266.7,173) and (262,168.3) .. (262,162.5) -- cycle ;
\draw   (262,210.5) .. controls (262,204.7) and (266.7,200) .. (272.5,200) .. controls (278.3,200) and (283,204.7) .. (283,210.5) .. controls (283,216.3) and (278.3,221) .. (272.5,221) .. controls (266.7,221) and (262,216.3) .. (262,210.5) -- cycle ;
\draw   (351,93.5) .. controls (351,87.7) and (355.7,83) .. (361.5,83) .. controls (367.3,83) and (372,87.7) .. (372,93.5) .. controls (372,99.3) and (367.3,104) .. (361.5,104) .. controls (355.7,104) and (351,99.3) .. (351,93.5) -- cycle ;
\draw   (351,145.5) .. controls (351,139.7) and (355.7,135) .. (361.5,135) .. controls (367.3,135) and (372,139.7) .. (372,145.5) .. controls (372,151.3) and (367.3,156) .. (361.5,156) .. controls (355.7,156) and (351,151.3) .. (351,145.5) -- cycle ;
\draw   (351,201.5) .. controls (351,195.7) and (355.7,191) .. (361.5,191) .. controls (367.3,191) and (372,195.7) .. (372,201.5) .. controls (372,207.3) and (367.3,212) .. (361.5,212) .. controls (355.7,212) and (351,207.3) .. (351,201.5) -- cycle ;
\draw    (283,74.5) -- (351,93.5) ;
\draw    (283,118.5) -- (351,93.5) ;
\draw    (283,162.5) -- (351,93.5) ;
\draw    (283,118.5) -- (351,145.5) ;
\draw    (283,162.5) -- (351,201.5) ;
\draw    (283,210.5) -- (351,201.5) ;
\draw    (283,210.5) -- (351,145.5) ;

\draw (265,67) node [anchor=north west][inner sep=0.75pt]  [font=\footnotesize] [align=left] {$\displaystyle v_{1}$};
\draw (265,111) node [anchor=north west][inner sep=0.75pt]  [font=\footnotesize] [align=left] {$\displaystyle v_{2}$};
\draw (265,155) node [anchor=north west][inner sep=0.75pt]  [font=\footnotesize] [align=left] {$\displaystyle v_{3}$};
\draw (265,203) node [anchor=north west][inner sep=0.75pt]  [font=\footnotesize] [align=left] {$\displaystyle v_{4}$};
\draw (354,86) node [anchor=north west][inner sep=0.75pt]  [font=\footnotesize] [align=left] {$\displaystyle h_{1}$};
\draw (354,138) node [anchor=north west][inner sep=0.75pt]  [font=\footnotesize] [align=left] {$\displaystyle h_{2}$};
\draw (354,194) node [anchor=north west][inner sep=0.75pt]  [font=\footnotesize] [align=left] {$\displaystyle h_{3}$};
\end{tikzpicture}
    \caption{The incidence graph of the hypergraph in Figure~\ref{fig:hypergraph-ex}.}
    \label{fig:incidence-graph-ex}
\end{figure}

\subsection{The zero forcing process on graphs}

The {zero forcing process} on a graph $G$ is defined as follows. Initially, there is a subset $B$ of blue vertices, while all other vertices are white. 
The \emph{color change rule} dictates that at each step, a blue vertex
with exactly one white neighbor will force its white neighbor to become blue. 
Each vertex will remain the same color from the previous step unless affected by the color change rule.
The set $B$ is said to be a zero forcing set if, by iteratively applying the color change rule, all of $V$ becomes blue. The \emph{zero forcing number} of $G$ is the minimum cardinality of a zero
forcing set in $G$, denoted by $Z(G)$.

The following result will be used in Section \ref{sec:burning}.

\begin{theorem}(cf. \cite{hogben2022inverse})
If $G$ is a disconnected graph with connected components $\{G_i\}_{1\leq i\leq k}$, then
    \[ Z(G)=\sum_{i=1}^{k}Z(G_i).
    \]
    \label{thm:connected components zero forcing number}
\end{theorem}

An interesting observation can be made on the vertices in a minimal zero forcing set. At least one neighbor of every vertex in the zero forcing set must initially be uncolored.
This observation was generalized for the $k$-forcing process in \cite{k-forcing-upper}. However, in this paper only the zero forcing process for $k=1$ will be used.

\begin{lemma}{\cite{k-forcing-upper}}
    Let $G$ be a graph of order $n\geq 2$ with a minimum degree $\delta\geq 1$. Let $B$ be a minimum zero forcing set of $G$. Then for each $v\in V(G)$, there is at least one vertex $w\in V(G)\backslash B$ that is adjacent to $v$.
    \label{lemma:zfp-neighbor-uncolored}
\end{lemma}

\subsection{The burning process on (hyper)graphs}

   Consider a graph $G=(V(G),E(G))$. The vertices of the graph $G$ can be either \textit{burned} or \textit{unburned} throughout the process. Initially, in the first round $t=0$, all vertices are unburned. In each round $t\geq 1$, an unburned vertex is chosen to become burned, if there is one, and is called a \textit{source}. An unburned vertex remains unburned unless it becomes burned by the burning process, and once a vertex becomes burned, it remains burned for the rest of the process. If a vertex is burned in round $t$, then all its neighbors become burned in round $t+1$. 
    The process ends when all vertices are burned. 
    The sources that are chosen over time form a \textit{burning sequence}, and the shortest of all possible burning sequences is called \textit{optimal}.
    Note that a new source is chosen first in each round and only later in the same round do the neighbors of the previously burned vertices become burned.
    A source in round $t$ is called \textit{redundant} if one of its neighbors is burned in round $t-1$.

In contrast to the firefighting problem, where the goal is to contain a fire, in the burning process, the aim is usually to spread an influence as quickly as possible over a network.
Therefore, a natural question to ask is which sources should be chosen and in which order they need to be chosen to form an optimal burning sequence so that the entire network has received the message in as few rounds as possible.



The \textit{burning number} $b(G)$ of a graph $G$ is the minimum number of rounds needed to burn the entire graph with the burning process.

\begin{example}
    Consider the graph $G=(V(G),E(G))$ in Figure~\ref{fig:graph burning}. The propagation of the burning process with burning sequence $S=(v_2,v_5,v_8)$ is shown. In the first round, only $v_2$ is burned, indicated with deep red. In the second round, first $v_5$ gets burned and then the neighbors of $v_2$ get burned, indicated with light red. In the third round, all the entire graphs is burned. In fact, the burning number of $G$ is $b(G)=3$ which means that $S$ is an optimal burning sequence. Note that $S$ is not a unique optimal burning sequence in this case as $S'=(v_6,v_1,v_5)$ is also a burning sequence of length 3.
    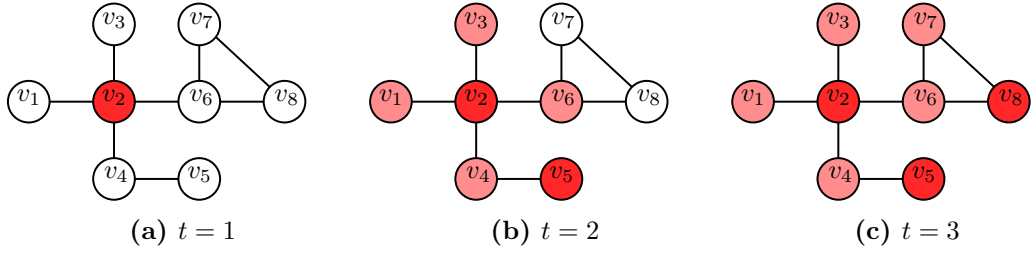
\begin{figure}[htp!]
        \centering
        \begin{subfigure}[t]{0.30\textwidth}

\tikzset{every picture/.style={line width=0.75pt}} 

\begin{tikzpicture}[x=0.75pt,y=0.75pt,yscale=-1,xscale=1]

\draw   (257,103.5) .. controls (257,97.7) and (261.7,93) .. (267.5,93) .. controls (273.3,93) and (278,97.7) .. (278,103.5) .. controls (278,109.3) and (273.3,114) .. (267.5,114) .. controls (261.7,114) and (257,109.3) .. (257,103.5) -- cycle ;
\draw  [fill={rgb, 255:red, 255; green, 0; blue, 0 }  ,fill opacity=0.85 ] (257,142.5) .. controls (257,136.7) and (261.7,132) .. (267.5,132) .. controls (273.3,132) and (278,136.7) .. (278,142.5) .. controls (278,148.3) and (273.3,153) .. (267.5,153) .. controls (261.7,153) and (257,148.3) .. (257,142.5) -- cycle ;
\draw   (257,181.5) .. controls (257,175.7) and (261.7,171) .. (267.5,171) .. controls (273.3,171) and (278,175.7) .. (278,181.5) .. controls (278,187.3) and (273.3,192) .. (267.5,192) .. controls (261.7,192) and (257,187.3) .. (257,181.5) -- cycle ;
\draw   (300,142.5) .. controls (300,136.7) and (304.7,132) .. (310.5,132) .. controls (316.3,132) and (321,136.7) .. (321,142.5) .. controls (321,148.3) and (316.3,153) .. (310.5,153) .. controls (304.7,153) and (300,148.3) .. (300,142.5) -- cycle ;
\draw   (214,142.5) .. controls (214,136.7) and (218.7,132) .. (224.5,132) .. controls (230.3,132) and (235,136.7) .. (235,142.5) .. controls (235,148.3) and (230.3,153) .. (224.5,153) .. controls (218.7,153) and (214,148.3) .. (214,142.5) -- cycle ;
\draw   (300,181.5) .. controls (300,175.7) and (304.7,171) .. (310.5,171) .. controls (316.3,171) and (321,175.7) .. (321,181.5) .. controls (321,187.3) and (316.3,192) .. (310.5,192) .. controls (304.7,192) and (300,187.3) .. (300,181.5) -- cycle ;
\draw   (343,142.5) .. controls (343,136.7) and (347.7,132) .. (353.5,132) .. controls (359.3,132) and (364,136.7) .. (364,142.5) .. controls (364,148.3) and (359.3,153) .. (353.5,153) .. controls (347.7,153) and (343,148.3) .. (343,142.5) -- cycle ;
\draw   (300,103.5) .. controls (300,97.7) and (304.7,93) .. (310.5,93) .. controls (316.3,93) and (321,97.7) .. (321,103.5) .. controls (321,109.3) and (316.3,114) .. (310.5,114) .. controls (304.7,114) and (300,109.3) .. (300,103.5) -- cycle ;
\draw    (267.5,153) -- (267.5,171) ;
\draw    (300,181.5) -- (278,181.5) ;
\draw [fill={rgb, 255:red, 255; green, 121; blue, 121 }  ,fill opacity=0.85 ]   (278,142.5) -- (300,142.5) ;
\draw [fill={rgb, 255:red, 255; green, 121; blue, 121 }  ,fill opacity=0.85 ]   (310.5,114) -- (310.5,132) ;
\draw    (321,142.5) -- (343,142.5) ;
\draw    (319,110) -- (346,135) ;
\draw    (267.5,114) -- (267.5,132) ;
\draw    (235,142.5) -- (257,142.5) ;

\draw (217,135) node [anchor=north west][inner sep=0.75pt]  [font=\footnotesize] [align=left] {$\displaystyle v_{1}$};
\draw (303,174) node [anchor=north west][inner sep=0.75pt]  [font=\footnotesize] [align=left] {$\displaystyle v_{5}$};
\draw (261,174) node [anchor=north west][inner sep=0.75pt]  [font=\footnotesize] [align=left] {$\displaystyle v_{4}$};
\draw (260,135) node [anchor=north west][inner sep=0.75pt]  [font=\footnotesize] [align=left] {$\displaystyle v_{2}$};
\draw (304,135) node [anchor=north west][inner sep=0.75pt]  [font=\footnotesize] [align=left] {$\displaystyle v_{6}$};
\draw (260,96) node [anchor=north west][inner sep=0.75pt]  [font=\footnotesize] [align=left] {$\displaystyle v_{3}$};
\draw (347,135) node [anchor=north west][inner sep=0.75pt]  [font=\footnotesize] [align=left] {$\displaystyle v_{8}$};
\draw (304,96) node [anchor=north west][inner sep=0.75pt]  [font=\footnotesize] [align=left] {$\displaystyle v_{7}$};

\end{tikzpicture}

            \caption{$t=1$}
        \end{subfigure}
        \begin{subfigure}[t]{0.30\textwidth}

\tikzset{every picture/.style={line width=0.75pt}} 

\begin{tikzpicture}[x=0.75pt,y=0.75pt,yscale=-1,xscale=1]

\draw  [fill={rgb, 255:red, 255; green, 121; blue, 121 }  ,fill opacity=0.85 ] (257,103.5) .. controls (257,97.7) and (261.7,93) .. (267.5,93) .. controls (273.3,93) and (278,97.7) .. (278,103.5) .. controls (278,109.3) and (273.3,114) .. (267.5,114) .. controls (261.7,114) and (257,109.3) .. (257,103.5) -- cycle ;
\draw  [fill={rgb, 255:red, 255; green, 0; blue, 0 }  ,fill opacity=0.85 ] (257,142.5) .. controls (257,136.7) and (261.7,132) .. (267.5,132) .. controls (273.3,132) and (278,136.7) .. (278,142.5) .. controls (278,148.3) and (273.3,153) .. (267.5,153) .. controls (261.7,153) and (257,148.3) .. (257,142.5) -- cycle ;
\draw  [fill={rgb, 255:red, 255; green, 121; blue, 121 }  ,fill opacity=0.85 ] (257,181.5) .. controls (257,175.7) and (261.7,171) .. (267.5,171) .. controls (273.3,171) and (278,175.7) .. (278,181.5) .. controls (278,187.3) and (273.3,192) .. (267.5,192) .. controls (261.7,192) and (257,187.3) .. (257,181.5) -- cycle ;
\draw  [fill={rgb, 255:red, 255; green, 121; blue, 121 }  ,fill opacity=0.85 ] (300,142.5) .. controls (300,136.7) and (304.7,132) .. (310.5,132) .. controls (316.3,132) and (321,136.7) .. (321,142.5) .. controls (321,148.3) and (316.3,153) .. (310.5,153) .. controls (304.7,153) and (300,148.3) .. (300,142.5) -- cycle ;
\draw  [fill={rgb, 255:red, 255; green, 121; blue, 121 }  ,fill opacity=0.85 ] (214,142.5) .. controls (214,136.7) and (218.7,132) .. (224.5,132) .. controls (230.3,132) and (235,136.7) .. (235,142.5) .. controls (235,148.3) and (230.3,153) .. (224.5,153) .. controls (218.7,153) and (214,148.3) .. (214,142.5) -- cycle ;
\draw  [fill={rgb, 255:red, 255; green, 0; blue, 0 }  ,fill opacity=0.85 ] (300,181.5) .. controls (300,175.7) and (304.7,171) .. (310.5,171) .. controls (316.3,171) and (321,175.7) .. (321,181.5) .. controls (321,187.3) and (316.3,192) .. (310.5,192) .. controls (304.7,192) and (300,187.3) .. (300,181.5) -- cycle ;
\draw   (343,142.5) .. controls (343,136.7) and (347.7,132) .. (353.5,132) .. controls (359.3,132) and (364,136.7) .. (364,142.5) .. controls (364,148.3) and (359.3,153) .. (353.5,153) .. controls (347.7,153) and (343,148.3) .. (343,142.5) -- cycle ;
\draw   (300,103.5) .. controls (300,97.7) and (304.7,93) .. (310.5,93) .. controls (316.3,93) and (321,97.7) .. (321,103.5) .. controls (321,109.3) and (316.3,114) .. (310.5,114) .. controls (304.7,114) and (300,109.3) .. (300,103.5) -- cycle ;
\draw [fill={rgb, 255:red, 255; green, 121; blue, 121 }  ,fill opacity=0.85 ]   (267.5,153) -- (267.5,171) ;
\draw [fill={rgb, 255:red, 255; green, 121; blue, 121 }  ,fill opacity=0.85 ]   (300,181.5) -- (278,181.5) ;
\draw [fill={rgb, 255:red, 255; green, 121; blue, 121 }  ,fill opacity=0.85 ]   (278,142.5) -- (300,142.5) ;
\draw [fill={rgb, 255:red, 255; green, 121; blue, 121 }  ,fill opacity=0.85 ]   (310.5,114) -- (310.5,132) ;
\draw    (321,142.5) -- (343,142.5) ;
\draw    (319,110) -- (346,135) ;
\draw [fill={rgb, 255:red, 255; green, 121; blue, 121 }  ,fill opacity=0.85 ]   (267.5,114) -- (267.5,132) ;
\draw [fill={rgb, 255:red, 255; green, 121; blue, 121 }  ,fill opacity=0.85 ]   (235,142.5) -- (257,142.5) ;

\draw (217,135) node [anchor=north west][inner sep=0.75pt]  [font=\footnotesize] [align=left] {$\displaystyle v_{1}$};
\draw (303,174) node [anchor=north west][inner sep=0.75pt]  [font=\footnotesize] [align=left] {$\displaystyle v_{5}$};
\draw (261,174) node [anchor=north west][inner sep=0.75pt]  [font=\footnotesize] [align=left] {$\displaystyle v_{4}$};
\draw (260,135) node [anchor=north west][inner sep=0.75pt]  [font=\footnotesize] [align=left] {$\displaystyle v_{2}$};
\draw (304,135) node [anchor=north west][inner sep=0.75pt]  [font=\footnotesize] [align=left] {$\displaystyle v_{6}$};
\draw (260,96) node [anchor=north west][inner sep=0.75pt]  [font=\footnotesize] [align=left] {$\displaystyle v_{3}$};
\draw (347,135) node [anchor=north west][inner sep=0.75pt]  [font=\footnotesize] [align=left] {$\displaystyle v_{8}$};
\draw (304,96) node [anchor=north west][inner sep=0.75pt]  [font=\footnotesize] [align=left] {$\displaystyle v_{7}$};

\end{tikzpicture}
            \caption{$t=2$}
        \end{subfigure}
        \begin{subfigure}[t]{0.30\textwidth}

\tikzset{every picture/.style={line width=0.75pt}} 

\begin{tikzpicture}[x=0.75pt,y=0.75pt,yscale=-1,xscale=1]

\draw  [fill={rgb, 255:red, 255; green, 121; blue, 121 }  ,fill opacity=0.85 ] (257,103.5) .. controls (257,97.7) and (261.7,93) .. (267.5,93) .. controls (273.3,93) and (278,97.7) .. (278,103.5) .. controls (278,109.3) and (273.3,114) .. (267.5,114) .. controls (261.7,114) and (257,109.3) .. (257,103.5) -- cycle ;
\draw  [fill={rgb, 255:red, 255; green, 0; blue, 0 }  ,fill opacity=0.85 ] (257,142.5) .. controls (257,136.7) and (261.7,132) .. (267.5,132) .. controls (273.3,132) and (278,136.7) .. (278,142.5) .. controls (278,148.3) and (273.3,153) .. (267.5,153) .. controls (261.7,153) and (257,148.3) .. (257,142.5) -- cycle ;
\draw  [fill={rgb, 255:red, 255; green, 121; blue, 121 }  ,fill opacity=0.85 ] (257,181.5) .. controls (257,175.7) and (261.7,171) .. (267.5,171) .. controls (273.3,171) and (278,175.7) .. (278,181.5) .. controls (278,187.3) and (273.3,192) .. (267.5,192) .. controls (261.7,192) and (257,187.3) .. (257,181.5) -- cycle ;
\draw  [fill={rgb, 255:red, 255; green, 121; blue, 121 }  ,fill opacity=0.85 ] (300,142.5) .. controls (300,136.7) and (304.7,132) .. (310.5,132) .. controls (316.3,132) and (321,136.7) .. (321,142.5) .. controls (321,148.3) and (316.3,153) .. (310.5,153) .. controls (304.7,153) and (300,148.3) .. (300,142.5) -- cycle ;
\draw  [fill={rgb, 255:red, 255; green, 121; blue, 121 }  ,fill opacity=0.85 ] (214,142.5) .. controls (214,136.7) and (218.7,132) .. (224.5,132) .. controls (230.3,132) and (235,136.7) .. (235,142.5) .. controls (235,148.3) and (230.3,153) .. (224.5,153) .. controls (218.7,153) and (214,148.3) .. (214,142.5) -- cycle ;
\draw  [fill={rgb, 255:red, 255; green, 0; blue, 0 }  ,fill opacity=0.85 ] (300,181.5) .. controls (300,175.7) and (304.7,171) .. (310.5,171) .. controls (316.3,171) and (321,175.7) .. (321,181.5) .. controls (321,187.3) and (316.3,192) .. (310.5,192) .. controls (304.7,192) and (300,187.3) .. (300,181.5) -- cycle ;
\draw  [fill={rgb, 255:red, 255; green, 0; blue, 0 }  ,fill opacity=0.85 ] (343,142.5) .. controls (343,136.7) and (347.7,132) .. (353.5,132) .. controls (359.3,132) and (364,136.7) .. (364,142.5) .. controls (364,148.3) and (359.3,153) .. (353.5,153) .. controls (347.7,153) and (343,148.3) .. (343,142.5) -- cycle ;
\draw  [fill={rgb, 255:red, 255; green, 121; blue, 121 }  ,fill opacity=0.85 ] (300,103.5) .. controls (300,97.7) and (304.7,93) .. (310.5,93) .. controls (316.3,93) and (321,97.7) .. (321,103.5) .. controls (321,109.3) and (316.3,114) .. (310.5,114) .. controls (304.7,114) and (300,109.3) .. (300,103.5) -- cycle ;
\draw [fill={rgb, 255:red, 255; green, 121; blue, 121 }  ,fill opacity=0.85 ]   (267.5,153) -- (267.5,171) ;
\draw [fill={rgb, 255:red, 255; green, 121; blue, 121 }  ,fill opacity=0.85 ]   (300,181.5) -- (278,181.5) ;
\draw [fill={rgb, 255:red, 255; green, 121; blue, 121 }  ,fill opacity=0.85 ]   (278,142.5) -- (300,142.5) ;
\draw [fill={rgb, 255:red, 255; green, 121; blue, 121 }  ,fill opacity=0.85 ]   (310.5,114) -- (310.5,132) ;
\draw [fill={rgb, 255:red, 255; green, 121; blue, 121 }  ,fill opacity=0.85 ]   (321,142.5) -- (343,142.5) ;
\draw [fill={rgb, 255:red, 255; green, 121; blue, 121 }  ,fill opacity=0.85 ]   (319,110) -- (346,135) ;
\draw [fill={rgb, 255:red, 255; green, 121; blue, 121 }  ,fill opacity=0.85 ]   (267.5,114) -- (267.5,132) ;
\draw [fill={rgb, 255:red, 255; green, 121; blue, 121 }  ,fill opacity=0.85 ]   (235,142.5) -- (257,142.5) ;

\draw (217,135) node [anchor=north west][inner sep=0.75pt]  [font=\footnotesize] [align=left] {$\displaystyle v_{1}$};
\draw (303,174) node [anchor=north west][inner sep=0.75pt]  [font=\footnotesize] [align=left] {$\displaystyle v_{5}$};
\draw (261,174) node [anchor=north west][inner sep=0.75pt]  [font=\footnotesize] [align=left] {$\displaystyle v_{4}$};
\draw (260,135) node [anchor=north west][inner sep=0.75pt]  [font=\footnotesize] [align=left] {$\displaystyle v_{2}$};
\draw (304,135) node [anchor=north west][inner sep=0.75pt]  [font=\footnotesize] [align=left] {$\displaystyle v_{6}$};
\draw (260,96) node [anchor=north west][inner sep=0.75pt]  [font=\footnotesize] [align=left] {$\displaystyle v_{3}$};
\draw (347,135) node [anchor=north west][inner sep=0.75pt]  [font=\footnotesize] [align=left] {$\displaystyle v_{8}$};
\draw (304,96) node [anchor=north west][inner sep=0.75pt]  [font=\footnotesize] [align=left] {$\displaystyle v_{7}$};

\end{tikzpicture}
            \caption{$t=3$}
        \end{subfigure}
        \caption{Graph burning process for the first three rounds with burned vertices marked in red and sources are highlighted in deep red.}
        \label{fig:graph burning}
    \end{figure}
\end{example}

Finding the burning number by checking every possible burning sequence could work for smaller graphs, but it becomes increasingly more difficult as the number of vertices grows. 
There exist other methods for finding the burning number of a graph that do not involve trying every possible burning sequence.
For that, other characterizations of the burning number are used.
In a graph $G$, observe that when a vertex $v_i$ is chosen as the source in round $i<b(G)$, in round $j>i$ all vertices in $N_{j-i}[v_i]$ are burned considering $b(G)\geq j$. 
Additionally, if a vertex is burned in round $i$, it cannot be chosen as the source in any round $i<j\leq b(G)$. This means that two sources $v_i$ and $v_j$ in a burning sequence must be at least a distance $|i-j|$ from each other.
These observations lead to another approach to find the burning number of a graph, which we will use in Section \ref{sec:cospectralburning}.

\begin{theorem}{(Alternative characterization of a burning sequence, \cite{survey-burning-number})}
    Let $G$ be a graph. A sequence $(v_1,v_2,\dots,v_k)$ of length $k$ in the graph $G$ is a burning sequence if and only if for each pair $i,j\in[k]$, it holds that $d(v_i,v_j)\geq |i-j|$ and
    $\cup_{l=1}^k N_{k-l}[v_l] = V(G)$.
    \label{thm:alternative burning}
\end{theorem}

The neighborhoods of the sources may have overlap, meaning in a burning sequence of length $k$ there might exist $i,j\in [k]$ such that $N_{k-i}[v_i] \cap N_{k-j}[v_j]\neq \emptyset$.
Additionally, consider a burning sequence $S$. If two vertices $v,w\in V(G)$ are burned in round $i$ and there is an edge $\{v,w\}\in E(G)$ connecting the two vertices, the edge has no influence on the burning number as $S$ remains a burning sequence.


The \textit{burning process in a hypergraph} $H$ is defined as follows. 
In each round $t$ an unburned vertex is chosen to become burned and is called a \textit{source}.
An unburned vertex $v$ becomes burned in round $t+1$ if and only if there is a non-singleton hyperedge $h\in E(H)$ with $v\in h$ such that $v$ was not burned, but the other vertices in $h$ are burned in round $t$.
Again, an unburned vertex will remain unburned unless it becomes burned by the burning process on hypergraphs, and once a vertex is burned, it will remain burned for the rest of the process.
The burning number $b(H)$ is the minimal number of rounds needed to burn the entire hypergraph.
The sources that are chosen over time form a burning sequence and the shortest of all possible burning sequences is called optimal.

Consider a hypergraph $H$ where all edges contain exactly two vertices. A vertex $v$ in $H$ becomes burned in the burning process at round $t+1$ if it is the only unburned vertex in a hyperedge. Since every hyperedge consists of exactly two vertices, $v$ becomes burned in round $t+1$ if any of its neighbors is burned in round $t$.
Note that this is exactly the same condition as in the definition of the burning process of a graph.
This means that hypergraph burning actually reduces to the burning of a graph when all edges contain exactly two vertices.
The relation between the burning number and zero forcing number on hypergraphs uses a slightly different notion of the burning process, namely the \textit{lazy burning process}.

A variation of the burning of a hypergraph is the \textit{lazy hypergraph burning}, introduced in \cite{hypergraph}.
Instead of choosing a new source in each round, a set of vertices $B\subset V(H)$ is chosen to burn only in the first round. In subsequent rounds, no other sources are chosen.
The burning of vertices remains the same as in the definition of the burning process in a hypergraph.
The set $B$ is a \textit{lazy burning set} if it eventually burns the entire hypergraph.
The \textit{lazy burning number} is the minimum cardinality of any lazy burning set and is denoted as $b_L(H)$. 
If $B$ is a lazy burning set and $|B|=b_L(H)$, then $B$ is called \textit{optimal}.

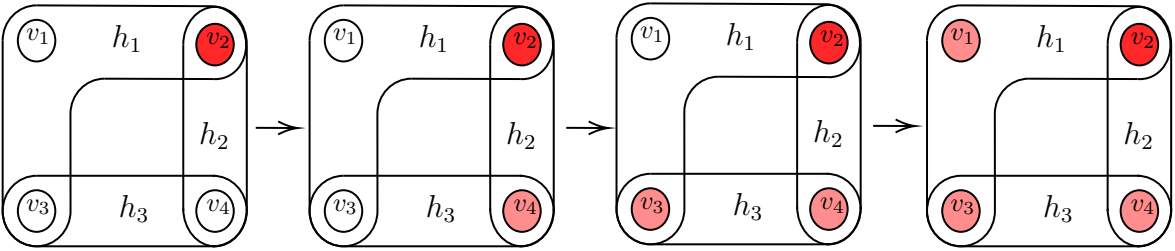
\begin{figure}[htp!]
	\centering

\tikzset{every picture/.style={line width=0.75pt}} 

\begin{tikzpicture}[x=0.75pt,y=0.75pt,yscale=-1,xscale=0.9]

\draw   (12,99.5) .. controls (12,93.7) and (16.7,89) .. (22.5,89) .. controls (28.3,89) and (33,93.7) .. (33,99.5) .. controls (33,105.3) and (28.3,110) .. (22.5,110) .. controls (16.7,110) and (12,105.3) .. (12,99.5) -- cycle ;
\draw   (12,185.5) .. controls (12,179.7) and (16.7,175) .. (22.5,175) .. controls (28.3,175) and (33,179.7) .. (33,185.5) .. controls (33,191.3) and (28.3,196) .. (22.5,196) .. controls (16.7,196) and (12,191.3) .. (12,185.5) -- cycle ;
\draw   (112,185.5) .. controls (112,179.7) and (116.7,175) .. (122.5,175) .. controls (128.3,175) and (133,179.7) .. (133,185.5) .. controls (133,191.3) and (128.3,196) .. (122.5,196) .. controls (116.7,196) and (112,191.3) .. (112,185.5) -- cycle ;
\draw  [fill={rgb, 255:red, 255; green, 0; blue, 0 }  ,fill opacity=0.85 ] (112,101.5) .. controls (112,95.7) and (116.7,91) .. (122.5,91) .. controls (128.3,91) and (133,95.7) .. (133,101.5) .. controls (133,107.3) and (128.3,112) .. (122.5,112) .. controls (116.7,112) and (112,107.3) .. (112,101.5) -- cycle ;
\draw    (22.5,167.7) -- (122.5,167.7) ;
\draw  [draw opacity=0] (121.27,167.7) .. controls (121.27,167.7) and (121.27,167.7) .. (121.27,167.7) .. controls (131.78,167.7) and (140.3,175.67) .. (140.3,185.5) .. controls (140.3,195.33) and (131.78,203.3) .. (121.27,203.3) -- (121.27,185.5) -- cycle ; \draw   (121.27,167.7) .. controls (121.27,167.7) and (121.27,167.7) .. (121.27,167.7) .. controls (131.78,167.7) and (140.3,175.67) .. (140.3,185.5) .. controls (140.3,195.33) and (131.78,203.3) .. (121.27,203.3) ;  
\draw    (22.5,203.3) -- (122.5,203.3) ;
\draw  [draw opacity=0] (140.3,184.27) .. controls (140.3,184.27) and (140.3,184.27) .. (140.3,184.27) .. controls (140.3,194.78) and (132.33,203.3) .. (122.5,203.3) .. controls (112.67,203.3) and (104.7,194.78) .. (104.7,184.27) -- (122.5,184.27) -- cycle ; \draw   (140.3,184.27) .. controls (140.3,184.27) and (140.3,184.27) .. (140.3,184.27) .. controls (140.3,194.78) and (132.33,203.3) .. (122.5,203.3) .. controls (112.67,203.3) and (104.7,194.78) .. (104.7,184.27) ;  
\draw    (104.7,101.5) -- (104.7,185.5) ;
\draw    (140.3,101.5) -- (140.3,185.5) ;
\draw  [draw opacity=0] (22.5,203.3) .. controls (22.5,203.3) and (22.5,203.3) .. (22.5,203.3) .. controls (22.5,203.3) and (22.5,203.3) .. (22.5,203.3) .. controls (11.99,203.3) and (3.48,195.33) .. (3.48,185.5) .. controls (3.48,175.67) and (11.99,167.7) .. (22.5,167.7) -- (22.5,185.5) -- cycle ; \draw   (22.5,203.3) .. controls (22.5,203.3) and (22.5,203.3) .. (22.5,203.3) .. controls (22.5,203.3) and (22.5,203.3) .. (22.5,203.3) .. controls (11.99,203.3) and (3.48,195.33) .. (3.48,185.5) .. controls (3.48,175.67) and (11.99,167.7) .. (22.5,167.7) ;  
\draw  [draw opacity=0] (104.7,101.5) .. controls (104.7,101.5) and (104.7,101.5) .. (104.7,101.5) .. controls (104.7,90.99) and (112.67,82.48) .. (122.5,82.48) .. controls (132.33,82.48) and (140.3,90.99) .. (140.3,101.5) -- (122.5,101.5) -- cycle ; \draw   (104.7,101.5) .. controls (104.7,101.5) and (104.7,101.5) .. (104.7,101.5) .. controls (104.7,90.99) and (112.67,82.48) .. (122.5,82.48) .. controls (132.33,82.48) and (140.3,90.99) .. (140.3,101.5) ;  
\draw  [draw opacity=0] (3.48,99.5) .. controls (3.48,89.67) and (11.99,81.7) .. (22.5,81.7) -- (22.5,99.5) -- cycle ; \draw   (3.48,99.5) .. controls (3.48,89.67) and (11.99,81.7) .. (22.5,81.7) ;  
\draw  [draw opacity=0] (41.52,185.5) .. controls (41.52,195.33) and (33.01,203.3) .. (22.5,203.3) .. controls (11.99,203.3) and (3.48,195.33) .. (3.48,185.5) -- (22.5,185.5) -- cycle ; \draw   (41.52,185.5) .. controls (41.52,195.33) and (33.01,203.3) .. (22.5,203.3) .. controls (11.99,203.3) and (3.48,195.33) .. (3.48,185.5) ;  
\draw    (3.48,99.5) -- (3.48,185.5) ;
\draw  [draw opacity=0] (121.5,82.12) .. controls (126.01,82.42) and (130.53,84.36) .. (134.09,87.91) .. controls (141.52,95.34) and (141.9,107) .. (134.95,113.95) .. controls (131.33,117.58) and (126.42,119.21) .. (121.5,118.88) -- (121.5,100.5) -- cycle ; \draw   (121.5,82.12) .. controls (126.01,82.42) and (130.53,84.36) .. (134.09,87.91) .. controls (141.52,95.34) and (141.9,107) .. (134.95,113.95) .. controls (131.33,117.58) and (126.42,119.21) .. (121.5,118.88) ;  
\draw    (22.5,81.7) -- (121.5,82.12) ;
\draw  [draw opacity=0] (41.48,136.5) .. controls (41.48,126.67) and (49.99,118.7) .. (60.5,118.7) -- (60.5,136.5) -- cycle ; \draw   (41.48,136.5) .. controls (41.48,126.67) and (49.99,118.7) .. (60.5,118.7) ;  
\draw    (41.48,136.5) -- (41.52,185.5) ;
\draw    (60.5,118.7) -- (121.5,118.88) ;
\draw    (145.3,143.5) -- (166.83,143.59) ;
\draw [shift={(168.82,143.6)}, rotate = 180.24] [color={rgb, 255:red, 0; green, 0; blue, 0 }  ][line width=0.75]    (10.93,-3.29) .. controls (6.95,-1.4) and (3.31,-0.3) .. (0,0) .. controls (3.31,0.3) and (6.95,1.4) .. (10.93,3.29)   ;
\draw   (184,99.5) .. controls (184,93.7) and (188.7,89) .. (194.5,89) .. controls (200.3,89) and (205,93.7) .. (205,99.5) .. controls (205,105.3) and (200.3,110) .. (194.5,110) .. controls (188.7,110) and (184,105.3) .. (184,99.5) -- cycle ;
\draw   (184,185.5) .. controls (184,179.7) and (188.7,175) .. (194.5,175) .. controls (200.3,175) and (205,179.7) .. (205,185.5) .. controls (205,191.3) and (200.3,196) .. (194.5,196) .. controls (188.7,196) and (184,191.3) .. (184,185.5) -- cycle ;
\draw  [fill={rgb, 255:red, 255; green, 121; blue, 121 }  ,fill opacity=0.85 ] (284,185.5) .. controls (284,179.7) and (288.7,175) .. (294.5,175) .. controls (300.3,175) and (305,179.7) .. (305,185.5) .. controls (305,191.3) and (300.3,196) .. (294.5,196) .. controls (288.7,196) and (284,191.3) .. (284,185.5) -- cycle ;
\draw  [fill={rgb, 255:red, 255; green, 0; blue, 0 }  ,fill opacity=0.85 ] (284,101.5) .. controls (284,95.7) and (288.7,91) .. (294.5,91) .. controls (300.3,91) and (305,95.7) .. (305,101.5) .. controls (305,107.3) and (300.3,112) .. (294.5,112) .. controls (288.7,112) and (284,107.3) .. (284,101.5) -- cycle ;
\draw    (194.5,167.7) -- (294.5,167.7) ;
\draw  [draw opacity=0] (293.27,167.7) .. controls (293.27,167.7) and (293.27,167.7) .. (293.27,167.7) .. controls (303.78,167.7) and (312.3,175.67) .. (312.3,185.5) .. controls (312.3,195.33) and (303.78,203.3) .. (293.27,203.3) -- (293.27,185.5) -- cycle ; \draw   (293.27,167.7) .. controls (293.27,167.7) and (293.27,167.7) .. (293.27,167.7) .. controls (303.78,167.7) and (312.3,175.67) .. (312.3,185.5) .. controls (312.3,195.33) and (303.78,203.3) .. (293.27,203.3) ;  
\draw    (194.5,203.3) -- (294.5,203.3) ;
\draw  [draw opacity=0] (312.3,184.27) .. controls (312.3,184.27) and (312.3,184.27) .. (312.3,184.27) .. controls (312.3,194.78) and (304.33,203.3) .. (294.5,203.3) .. controls (284.67,203.3) and (276.7,194.78) .. (276.7,184.27) -- (294.5,184.27) -- cycle ; \draw   (312.3,184.27) .. controls (312.3,184.27) and (312.3,184.27) .. (312.3,184.27) .. controls (312.3,194.78) and (304.33,203.3) .. (294.5,203.3) .. controls (284.67,203.3) and (276.7,194.78) .. (276.7,184.27) ;  
\draw    (276.7,101.5) -- (276.7,185.5) ;
\draw    (312.3,101.5) -- (312.3,185.5) ;
\draw  [draw opacity=0] (194.5,203.3) .. controls (194.5,203.3) and (194.5,203.3) .. (194.5,203.3) .. controls (183.99,203.3) and (175.48,195.33) .. (175.48,185.5) .. controls (175.48,175.67) and (183.99,167.7) .. (194.5,167.7) -- (194.5,185.5) -- cycle ; \draw   (194.5,203.3) .. controls (194.5,203.3) and (194.5,203.3) .. (194.5,203.3) .. controls (183.99,203.3) and (175.48,195.33) .. (175.48,185.5) .. controls (175.48,175.67) and (183.99,167.7) .. (194.5,167.7) ;  
\draw  [draw opacity=0] (276.7,101.5) .. controls (276.7,101.5) and (276.7,101.5) .. (276.7,101.5) .. controls (276.7,90.99) and (284.67,82.48) .. (294.5,82.48) .. controls (304.33,82.48) and (312.3,90.99) .. (312.3,101.5) -- (294.5,101.5) -- cycle ; \draw   (276.7,101.5) .. controls (276.7,101.5) and (276.7,101.5) .. (276.7,101.5) .. controls (276.7,90.99) and (284.67,82.48) .. (294.5,82.48) .. controls (304.33,82.48) and (312.3,90.99) .. (312.3,101.5) ;  
\draw  [draw opacity=0] (175.48,99.5) .. controls (175.48,89.67) and (183.99,81.7) .. (194.5,81.7) -- (194.5,99.5) -- cycle ; \draw   (175.48,99.5) .. controls (175.48,89.67) and (183.99,81.7) .. (194.5,81.7) ;  
\draw  [draw opacity=0] (213.52,185.5) .. controls (213.52,195.33) and (205.01,203.3) .. (194.5,203.3) .. controls (183.99,203.3) and (175.48,195.33) .. (175.48,185.5) -- (194.5,185.5) -- cycle ; \draw   (213.52,185.5) .. controls (213.52,195.33) and (205.01,203.3) .. (194.5,203.3) .. controls (183.99,203.3) and (175.48,195.33) .. (175.48,185.5) ;  
\draw    (175.48,99.5) -- (175.48,185.5) ;
\draw  [draw opacity=0] (293.5,82.12) .. controls (298.01,82.42) and (302.53,84.36) .. (306.09,87.91) .. controls (313.52,95.34) and (313.9,107) .. (306.95,113.95) .. controls (303.33,117.58) and (298.42,119.21) .. (293.5,118.88) -- (293.5,100.5) -- cycle ; \draw   (293.5,82.12) .. controls (298.01,82.42) and (302.53,84.36) .. (306.09,87.91) .. controls (313.52,95.34) and (313.9,107) .. (306.95,113.95) .. controls (303.33,117.58) and (298.42,119.21) .. (293.5,118.88) ;  
\draw    (194.5,81.7) -- (293.5,82.12) ;
\draw  [draw opacity=0] (213.48,136.5) .. controls (213.48,126.67) and (221.99,118.7) .. (232.5,118.7) -- (232.5,136.5) -- cycle ; \draw   (213.48,136.5) .. controls (213.48,126.67) and (221.99,118.7) .. (232.5,118.7) ;  
\draw    (213.48,136.5) -- (213.52,185.5) ;
\draw    (232.5,118.7) -- (293.5,118.88) ;
\draw   (356,98.5) .. controls (356,92.7) and (360.7,88) .. (366.5,88) .. controls (372.3,88) and (377,92.7) .. (377,98.5) .. controls (377,104.3) and (372.3,109) .. (366.5,109) .. controls (360.7,109) and (356,104.3) .. (356,98.5) -- cycle ;
\draw  [fill={rgb, 255:red, 255; green, 121; blue, 121 }  ,fill opacity=0.85 ] (356,184.5) .. controls (356,178.7) and (360.7,174) .. (366.5,174) .. controls (372.3,174) and (377,178.7) .. (377,184.5) .. controls (377,190.3) and (372.3,195) .. (366.5,195) .. controls (360.7,195) and (356,190.3) .. (356,184.5) -- cycle ;
\draw  [fill={rgb, 255:red, 255; green, 121; blue, 121 }  ,fill opacity=0.85 ] (456,184.5) .. controls (456,178.7) and (460.7,174) .. (466.5,174) .. controls (472.3,174) and (477,178.7) .. (477,184.5) .. controls (477,190.3) and (472.3,195) .. (466.5,195) .. controls (460.7,195) and (456,190.3) .. (456,184.5) -- cycle ;
\draw  [fill={rgb, 255:red, 255; green, 0; blue, 0 }  ,fill opacity=0.85 ] (456,100.5) .. controls (456,94.7) and (460.7,90) .. (466.5,90) .. controls (472.3,90) and (477,94.7) .. (477,100.5) .. controls (477,106.3) and (472.3,111) .. (466.5,111) .. controls (460.7,111) and (456,106.3) .. (456,100.5) -- cycle ;
\draw    (366.5,166.7) -- (466.5,166.7) ;
\draw  [draw opacity=0] (465.27,166.7) .. controls (465.27,166.7) and (465.27,166.7) .. (465.27,166.7) .. controls (475.78,166.7) and (484.3,174.67) .. (484.3,184.5) .. controls (484.3,194.33) and (475.78,202.3) .. (465.27,202.3) -- (465.27,184.5) -- cycle ; \draw   (465.27,166.7) .. controls (465.27,166.7) and (465.27,166.7) .. (465.27,166.7) .. controls (475.78,166.7) and (484.3,174.67) .. (484.3,184.5) .. controls (484.3,194.33) and (475.78,202.3) .. (465.27,202.3) ;  
\draw    (366.5,202.3) -- (466.5,202.3) ;
\draw  [draw opacity=0] (484.3,183.27) .. controls (484.3,183.27) and (484.3,183.27) .. (484.3,183.27) .. controls (484.3,193.78) and (476.33,202.3) .. (466.5,202.3) .. controls (456.67,202.3) and (448.7,193.78) .. (448.7,183.27) -- (466.5,183.27) -- cycle ; \draw   (484.3,183.27) .. controls (484.3,183.27) and (484.3,183.27) .. (484.3,183.27) .. controls (484.3,193.78) and (476.33,202.3) .. (466.5,202.3) .. controls (456.67,202.3) and (448.7,193.78) .. (448.7,183.27) ;  
\draw    (448.7,100.5) -- (448.7,184.5) ;
\draw    (484.3,100.5) -- (484.3,184.5) ;
\draw  [draw opacity=0] (366.5,202.3) .. controls (366.5,202.3) and (366.5,202.3) .. (366.5,202.3) .. controls (355.99,202.3) and (347.48,194.33) .. (347.48,184.5) .. controls (347.48,174.67) and (355.99,166.7) .. (366.5,166.7) -- (366.5,184.5) -- cycle ; \draw   (366.5,202.3) .. controls (366.5,202.3) and (366.5,202.3) .. (366.5,202.3) .. controls (355.99,202.3) and (347.48,194.33) .. (347.48,184.5) .. controls (347.48,174.67) and (355.99,166.7) .. (366.5,166.7) ;  
\draw  [draw opacity=0] (448.7,100.5) .. controls (448.7,100.5) and (448.7,100.5) .. (448.7,100.5) .. controls (448.7,89.99) and (456.67,81.48) .. (466.5,81.48) .. controls (476.33,81.48) and (484.3,89.99) .. (484.3,100.5) -- (466.5,100.5) -- cycle ; \draw   (448.7,100.5) .. controls (448.7,100.5) and (448.7,100.5) .. (448.7,100.5) .. controls (448.7,89.99) and (456.67,81.48) .. (466.5,81.48) .. controls (476.33,81.48) and (484.3,89.99) .. (484.3,100.5) ;  
\draw  [draw opacity=0] (347.48,98.5) .. controls (347.48,88.67) and (355.99,80.7) .. (366.5,80.7) -- (366.5,98.5) -- cycle ; \draw   (347.48,98.5) .. controls (347.48,88.67) and (355.99,80.7) .. (366.5,80.7) ;  
\draw  [draw opacity=0] (385.52,184.5) .. controls (385.52,194.33) and (377.01,202.3) .. (366.5,202.3) .. controls (355.99,202.3) and (347.48,194.33) .. (347.48,184.5) -- (366.5,184.5) -- cycle ; \draw   (385.52,184.5) .. controls (385.52,194.33) and (377.01,202.3) .. (366.5,202.3) .. controls (355.99,202.3) and (347.48,194.33) .. (347.48,184.5) ;  
\draw    (347.48,98.5) -- (347.48,184.5) ;
\draw  [draw opacity=0] (465.5,81.12) .. controls (470.01,81.42) and (474.53,83.36) .. (478.09,86.91) .. controls (485.52,94.34) and (485.9,106) .. (478.95,112.95) .. controls (475.33,116.58) and (470.42,118.21) .. (465.5,117.88) -- (465.5,99.5) -- cycle ; \draw   (465.5,81.12) .. controls (470.01,81.42) and (474.53,83.36) .. (478.09,86.91) .. controls (485.52,94.34) and (485.9,106) .. (478.95,112.95) .. controls (475.33,116.58) and (470.42,118.21) .. (465.5,117.88) ;  
\draw    (366.5,80.7) -- (465.5,81.12) ;
\draw  [draw opacity=0] (385.48,135.5) .. controls (385.48,125.67) and (393.99,117.7) .. (404.5,117.7) -- (404.5,135.5) -- cycle ; \draw   (385.48,135.5) .. controls (385.48,125.67) and (393.99,117.7) .. (404.5,117.7) ;  
\draw    (385.48,135.5) -- (385.52,184.5) ;
\draw    (404.5,117.7) -- (465.5,117.88) ;
\draw  [fill={rgb, 255:red, 255; green, 121; blue, 121 }  ,fill opacity=0.85 ] (530,99.5) .. controls (530,93.7) and (534.7,89) .. (540.5,89) .. controls (546.3,89) and (551,93.7) .. (551,99.5) .. controls (551,105.3) and (546.3,110) .. (540.5,110) .. controls (534.7,110) and (530,105.3) .. (530,99.5) -- cycle ;
\draw  [fill={rgb, 255:red, 255; green, 121; blue, 121 }  ,fill opacity=0.85 ] (530,185.5) .. controls (530,179.7) and (534.7,175) .. (540.5,175) .. controls (546.3,175) and (551,179.7) .. (551,185.5) .. controls (551,191.3) and (546.3,196) .. (540.5,196) .. controls (534.7,196) and (530,191.3) .. (530,185.5) -- cycle ;
\draw  [fill={rgb, 255:red, 255; green, 121; blue, 121 }  ,fill opacity=0.85 ] (630,185.5) .. controls (630,179.7) and (634.7,175) .. (640.5,175) .. controls (646.3,175) and (651,179.7) .. (651,185.5) .. controls (651,191.3) and (646.3,196) .. (640.5,196) .. controls (634.7,196) and (630,191.3) .. (630,185.5) -- cycle ;
\draw  [fill={rgb, 255:red, 255; green, 0; blue, 0 }  ,fill opacity=0.85 ] (630,101.5) .. controls (630,95.7) and (634.7,91) .. (640.5,91) .. controls (646.3,91) and (651,95.7) .. (651,101.5) .. controls (651,107.3) and (646.3,112) .. (640.5,112) .. controls (634.7,112) and (630,107.3) .. (630,101.5) -- cycle ;
\draw    (540.5,167.7) -- (640.5,167.7) ;
\draw  [draw opacity=0] (639.27,167.7) .. controls (639.27,167.7) and (639.27,167.7) .. (639.27,167.7) .. controls (649.78,167.7) and (658.3,175.67) .. (658.3,185.5) .. controls (658.3,195.33) and (649.78,203.3) .. (639.27,203.3) -- (639.27,185.5) -- cycle ; \draw   (639.27,167.7) .. controls (639.27,167.7) and (639.27,167.7) .. (639.27,167.7) .. controls (649.78,167.7) and (658.3,175.67) .. (658.3,185.5) .. controls (658.3,195.33) and (649.78,203.3) .. (639.27,203.3) ;  
\draw    (540.5,203.3) -- (640.5,203.3) ;
\draw  [draw opacity=0] (658.3,184.27) .. controls (658.3,184.27) and (658.3,184.27) .. (658.3,184.27) .. controls (658.3,194.78) and (650.33,203.3) .. (640.5,203.3) .. controls (630.67,203.3) and (622.7,194.78) .. (622.7,184.27) -- (640.5,184.27) -- cycle ; \draw   (658.3,184.27) .. controls (658.3,184.27) and (658.3,184.27) .. (658.3,184.27) .. controls (658.3,194.78) and (650.33,203.3) .. (640.5,203.3) .. controls (630.67,203.3) and (622.7,194.78) .. (622.7,184.27) ;  
\draw    (622.7,101.5) -- (622.7,185.5) ;
\draw    (658.3,101.5) -- (658.3,185.5) ;
\draw  [draw opacity=0] (540.5,203.3) .. controls (540.5,203.3) and (540.5,203.3) .. (540.5,203.3) .. controls (529.99,203.3) and (521.48,195.33) .. (521.48,185.5) .. controls (521.48,175.67) and (529.99,167.7) .. (540.5,167.7) -- (540.5,185.5) -- cycle ; \draw   (540.5,203.3) .. controls (540.5,203.3) and (540.5,203.3) .. (540.5,203.3) .. controls (529.99,203.3) and (521.48,195.33) .. (521.48,185.5) .. controls (521.48,175.67) and (529.99,167.7) .. (540.5,167.7) ;  
\draw  [draw opacity=0] (622.7,101.5) .. controls (622.7,101.5) and (622.7,101.5) .. (622.7,101.5) .. controls (622.7,90.99) and (630.67,82.48) .. (640.5,82.48) .. controls (650.33,82.48) and (658.3,90.99) .. (658.3,101.5) -- (640.5,101.5) -- cycle ; \draw   (622.7,101.5) .. controls (622.7,101.5) and (622.7,101.5) .. (622.7,101.5) .. controls (622.7,90.99) and (630.67,82.48) .. (640.5,82.48) .. controls (650.33,82.48) and (658.3,90.99) .. (658.3,101.5) ;  
\draw  [draw opacity=0] (521.48,99.5) .. controls (521.48,89.67) and (529.99,81.7) .. (540.5,81.7) -- (540.5,99.5) -- cycle ; \draw   (521.48,99.5) .. controls (521.48,89.67) and (529.99,81.7) .. (540.5,81.7) ;  
\draw  [draw opacity=0] (559.52,185.5) .. controls (559.52,195.33) and (551.01,203.3) .. (540.5,203.3) .. controls (529.99,203.3) and (521.48,195.33) .. (521.48,185.5) -- (540.5,185.5) -- cycle ; \draw   (559.52,185.5) .. controls (559.52,195.33) and (551.01,203.3) .. (540.5,203.3) .. controls (529.99,203.3) and (521.48,195.33) .. (521.48,185.5) ;  
\draw    (521.48,99.5) -- (521.48,185.5) ;
\draw  [draw opacity=0] (639.5,82.12) .. controls (644.01,82.42) and (648.53,84.36) .. (652.09,87.91) .. controls (659.52,95.34) and (659.9,107) .. (652.95,113.95) .. controls (649.33,117.58) and (644.42,119.21) .. (639.5,118.88) -- (639.5,100.5) -- cycle ; \draw   (639.5,82.12) .. controls (644.01,82.42) and (648.53,84.36) .. (652.09,87.91) .. controls (659.52,95.34) and (659.9,107) .. (652.95,113.95) .. controls (649.33,117.58) and (644.42,119.21) .. (639.5,118.88) ;  
\draw    (540.5,81.7) -- (639.5,82.12) ;
\draw  [draw opacity=0] (559.48,136.5) .. controls (559.48,126.67) and (567.99,118.7) .. (578.5,118.7) -- (578.5,136.5) -- cycle ; \draw   (559.48,136.5) .. controls (559.48,126.67) and (567.99,118.7) .. (578.5,118.7) ;  
\draw    (559.48,136.5) -- (559.52,185.5) ;
\draw    (578.5,118.7) -- (639.5,118.88) ;
\draw    (319.3,143.5) -- (340.83,143.59) ;
\draw [shift={(342.82,143.6)}, rotate = 180.24] [color={rgb, 255:red, 0; green, 0; blue, 0 }  ][line width=0.75]    (10.93,-3.29) .. controls (6.95,-1.4) and (3.31,-0.3) .. (0,0) .. controls (3.31,0.3) and (6.95,1.4) .. (10.93,3.29)   ;
\draw    (491.3,142.5) -- (512.83,142.59) ;
\draw [shift={(514.82,142.6)}, rotate = 180.24] [color={rgb, 255:red, 0; green, 0; blue, 0 }  ][line width=0.75]    (10.93,-3.29) .. controls (6.95,-1.4) and (3.31,-0.3) .. (0,0) .. controls (3.31,0.3) and (6.95,1.4) .. (10.93,3.29)   ;

\draw (15,178) node [anchor=north west][inner sep=0.75pt]  [font=\footnotesize] [align=left] {$\displaystyle v_{3}$};
\draw (116,178) node [anchor=north west][inner sep=0.75pt]  [font=\footnotesize] [align=left] {$\displaystyle v_{4}$};
\draw (15,92) node [anchor=north west][inner sep=0.75pt]  [font=\footnotesize] [align=left] {$\displaystyle v_{1}$};
\draw (116,94) node [anchor=north west][inner sep=0.75pt]  [font=\footnotesize] [align=left] {$\displaystyle v_{2}$};
\draw (63,90) node [anchor=north west][inner sep=0.75pt]   [align=left] {$\displaystyle h_{1}$};
\draw (112,139) node [anchor=north west][inner sep=0.75pt]   [align=left] {$\displaystyle h_{2}$};
\draw (67,176) node [anchor=north west][inner sep=0.75pt]   [align=left] {$\displaystyle h_{3}$};
\draw (187,178) node [anchor=north west][inner sep=0.75pt]  [font=\footnotesize] [align=left] {$\displaystyle v_{3}$};
\draw (288,178) node [anchor=north west][inner sep=0.75pt]  [font=\footnotesize] [align=left] {$\displaystyle v_{4}$};
\draw (187,92) node [anchor=north west][inner sep=0.75pt]  [font=\footnotesize] [align=left] {$\displaystyle v_{1}$};
\draw (288,94) node [anchor=north west][inner sep=0.75pt]  [font=\footnotesize] [align=left] {$\displaystyle v_{2}$};
\draw (235,90) node [anchor=north west][inner sep=0.75pt]   [align=left] {$\displaystyle h_{1}$};
\draw (284,139) node [anchor=north west][inner sep=0.75pt]   [align=left] {$\displaystyle h_{2}$};
\draw (239,176) node [anchor=north west][inner sep=0.75pt]   [align=left] {$\displaystyle h_{3}$};
\draw (359,177) node [anchor=north west][inner sep=0.75pt]  [font=\footnotesize] [align=left] {$\displaystyle v_{3}$};
\draw (460,177) node [anchor=north west][inner sep=0.75pt]  [font=\footnotesize] [align=left] {$\displaystyle v_{4}$};
\draw (359,91) node [anchor=north west][inner sep=0.75pt]  [font=\footnotesize] [align=left] {$\displaystyle v_{1}$};
\draw (460,93) node [anchor=north west][inner sep=0.75pt]  [font=\footnotesize] [align=left] {$\displaystyle v_{2}$};
\draw (407,89) node [anchor=north west][inner sep=0.75pt]   [align=left] {$\displaystyle h_{1}$};
\draw (456,138) node [anchor=north west][inner sep=0.75pt]   [align=left] {$\displaystyle h_{2}$};
\draw (411,175) node [anchor=north west][inner sep=0.75pt]   [align=left] {$\displaystyle h_{3}$};
\draw (533,178) node [anchor=north west][inner sep=0.75pt]  [font=\footnotesize] [align=left] {$\displaystyle v_{3}$};
\draw (634,178) node [anchor=north west][inner sep=0.75pt]  [font=\footnotesize] [align=left] {$\displaystyle v_{4}$};
\draw (533,92) node [anchor=north west][inner sep=0.75pt]  [font=\footnotesize] [align=left] {$\displaystyle v_{1}$};
\draw (634,94) node [anchor=north west][inner sep=0.75pt]  [font=\footnotesize] [align=left] {$\displaystyle v_{2}$};
\draw (581,90) node [anchor=north west][inner sep=0.75pt]   [align=left] {$\displaystyle h_{1}$};
\draw (630,139) node [anchor=north west][inner sep=0.75pt]   [align=left] {$\displaystyle h_{2}$};
\draw (585,176) node [anchor=north west][inner sep=0.75pt]   [align=left] {$\displaystyle h_{3}$};

\end{tikzpicture}

	\caption{The lazy burning process on a hypergraph.}
	\label{fig:lazy-burning-process-hypergraph}
\end{figure}

\begin{example}
	Consider the hypergraph $H$ in Figure~\ref{fig:hypergraph-ex}.
	If we take $B=\{v_2\}$, then vertex $v_4$ is burned by the hyperedge $h_2$. In the next round $v_3$ is burned by $h_3$ and in the round thereafter $v_1$ is burned by $h_1$. Since all hyperedges link at least two vertices, at least one initial burned vertex is needed to burn the hypergraph. Hence, in this case, $b_L(H) = 1$.
	This lazy burning process is illustrated in Figure~\ref{fig:lazy-burning-process-hypergraph}.
\end{example}

Using a similar reasoning as to why the hypergraph burning process reduces to the graph burning process, the lazy burning process almost reduces to the graph burning process if all hyperedges of a hypergraph $H$ contain exactly two vertices.
The difference, as explained before, is that sources do not appear each round, but a set of vertices is chosen to be burned in the first round already in the lazy burning process.
As a result of this, let $\{H_i\}_{1\leq i\leq k}$ be the set of the $k$ connected components of $H$, then $b_L(H_i)=1$ for each $1\leq i \leq k$ and thus $b_L(H)=\sum_{i=1}^k b_L(H_i)= k$.

Additionally, it is easy to see that the lazy burning number is a lower bound of the burning number of a hypergraph. 
Indeed, consider any hypergraph $H$ and let $(v_1,v_2,\dots,v_k)$ be an optimal burning sequence in $H$. Then, since this sequence burns all the vertices of $H$, the set $B=\{v_1,v_2,\dots,v_k\}$ is a lazy burning set. Some vertices in $B$ may be redundant to burn the entire hypergraph. This leads to the following inequality:

\begin{equation}
	b_L(H) \leq b(H).
	\label{eq:burning number vs lazy burning number}
\end{equation}

\section{Connection between the burning and the zero forcing processes via hypergraphs}\label{sec:burning}

The principal goal of this section is to show a sharp lower bound on the lazy burning number of a hypergraphs in terms of the zero forcing of the incidence graph of such hypergraph. Our result builds on work by Bonato et al \cite{hypergraph}, who recently provided a characterization of lazy burning sets in hypergraphs via zero forcing sets in their incidence graphs.





\begin{theorem}{\cite[Theorem 25]{hypergraph}}
   For a hypergraph $H$, a subset $B\subset V(H)$ is a lazy burning set for $H$ if and only if $B\cup E(H)$ is a zero forcing set for $IG(H)$. Then,
   \begin{equation}\label{eq:hypergraph inequality}
           Z(IG(H)) \leq b_L(H) + |E(H)|.
   \end{equation}
   \label{thm:hypergraph bound}
\end{theorem}
\vspace{-0.9cm}

\begin{example}
    Consider the hypergraph $H$ with its incidence graph $IG(H)$ in Figure~\ref{fig:hypergraph-ex} and Figure \ref{fig:incidence-graph-ex}, respectively. It is straightforward to check that the zero forcing number cannot be $1$. 
    Now, take the set $B=\{v_1,v_2\}$ for the zero forcing process. Then the following zero forcing chain is obtained $(v_1\rightarrow h_1,h_1\rightarrow v_3,v_2\rightarrow h_2,v_3\rightarrow h_3,h_2\rightarrow v_4)$, which means that $B$ is a zero forcing set, and we can conclude that $Z(IG(H))=2$. 
    Since the lazy burning number is $b_L(H)=1$, then inequality~\eqref{eq:hypergraph inequality} gives $2\leq 1+3$.
\end{example}

Consider again the hypergraph $H$ with incidence graph $IG(H)$ shown in Figure~\ref{fig:incidence-graph-ex}. 
Following the construction of the inequality in Theorem~\ref{thm:hypergraph bound}, the set $B=\{v_2\}\cup E(H)$ is a zero forcing set of $IG(H)$ as $\{v_2\}$ was a lazy burning set.
Observe that $v_2$ is only connected to vertices in $E(H)$ which are all colored already. 
However, by Lemma~\ref{lemma:zfp-neighbor-uncolored}, $B$ is not a minimum zero forcing set. Using a similar argument as in the proof of Lemma~\ref{lemma:zfp-neighbor-uncolored}, a vertex in $E(H)$ that is connected to $v_2$ can be removed and a smaller zero forcing set is obtained. Consider for example the set $B'=B\backslash \{h_2\}$. Then it can easily be verified that this is indeed a zero forcing set. 
Hence, the relation $Z(IG(H))\leq b_L(H) + |E(H)|-1$ holds for this graph.
This is shown in Figure~\ref{fig:smaller-zero-forcing-set}. 
This means that we can construct a smaller zero forcing set for the incidence graph $IG(H)$ then in the proof of Theorem~\ref{thm:hypergraph bound}.

\begin{figure}[htp!]
    \centering

\tikzset{every picture/.style={line width=0.75pt}} 

\begin{tikzpicture}[x=0.75pt,y=0.75pt,yscale=-1,xscale=1]

\draw   (130,71.5) .. controls (130,65.7) and (134.7,61) .. (140.5,61) .. controls (146.3,61) and (151,65.7) .. (151,71.5) .. controls (151,77.3) and (146.3,82) .. (140.5,82) .. controls (134.7,82) and (130,77.3) .. (130,71.5) -- cycle ;
\draw   (130,115.5) .. controls (130,109.7) and (134.7,105) .. (140.5,105) .. controls (146.3,105) and (151,109.7) .. (151,115.5) .. controls (151,121.3) and (146.3,126) .. (140.5,126) .. controls (134.7,126) and (130,121.3) .. (130,115.5) -- cycle ;
\draw   (130,159.5) .. controls (130,153.7) and (134.7,149) .. (140.5,149) .. controls (146.3,149) and (151,153.7) .. (151,159.5) .. controls (151,165.3) and (146.3,170) .. (140.5,170) .. controls (134.7,170) and (130,165.3) .. (130,159.5) -- cycle ;
\draw  [fill={rgb, 255:red, 74; green, 144; blue, 226 }  ,fill opacity=0.85 ] (130,207.5) .. controls (130,201.7) and (134.7,197) .. (140.5,197) .. controls (146.3,197) and (151,201.7) .. (151,207.5) .. controls (151,213.3) and (146.3,218) .. (140.5,218) .. controls (134.7,218) and (130,213.3) .. (130,207.5) -- cycle ;
\draw  [fill={rgb, 255:red, 74; green, 144; blue, 226 }  ,fill opacity=0.85 ] (219,90.5) .. controls (219,84.7) and (223.7,80) .. (229.5,80) .. controls (235.3,80) and (240,84.7) .. (240,90.5) .. controls (240,96.3) and (235.3,101) .. (229.5,101) .. controls (223.7,101) and (219,96.3) .. (219,90.5) -- cycle ;
\draw  [fill={rgb, 255:red, 74; green, 144; blue, 226 }  ,fill opacity=0.85 ] (219,142.5) .. controls (219,136.7) and (223.7,132) .. (229.5,132) .. controls (235.3,132) and (240,136.7) .. (240,142.5) .. controls (240,148.3) and (235.3,153) .. (229.5,153) .. controls (223.7,153) and (219,148.3) .. (219,142.5) -- cycle ;
\draw   (219,198.5) .. controls (219,192.7) and (223.7,188) .. (229.5,188) .. controls (235.3,188) and (240,192.7) .. (240,198.5) .. controls (240,204.3) and (235.3,209) .. (229.5,209) .. controls (223.7,209) and (219,204.3) .. (219,198.5) -- cycle ;
\draw [fill={rgb, 255:red, 74; green, 144; blue, 226 }  ,fill opacity=0.85 ]   (151,71.5) -- (219,90.5) ;
\draw [fill={rgb, 255:red, 74; green, 144; blue, 226 }  ,fill opacity=0.85 ]   (151,115.5) -- (219,90.5) ;
\draw [fill={rgb, 255:red, 74; green, 144; blue, 226 }  ,fill opacity=0.85 ]   (151,159.5) -- (219,90.5) ;
\draw [fill={rgb, 255:red, 74; green, 144; blue, 226 }  ,fill opacity=0.85 ]   (151,115.5) -- (219,142.5) ;
\draw [fill={rgb, 255:red, 74; green, 144; blue, 226 }  ,fill opacity=0.85 ]   (151,159.5) -- (219,198.5) ;
\draw [fill={rgb, 255:red, 74; green, 144; blue, 226 }  ,fill opacity=0.85 ]   (151,207.5) -- (219,198.5) ;
\draw [fill={rgb, 255:red, 74; green, 144; blue, 226 }  ,fill opacity=0.85 ]   (151,207.5) -- (219,142.5) ;
\draw    (260,142.5) -- (302,142.6) ;
\draw [shift={(304,142.6)}, rotate = 180.13] [color={rgb, 255:red, 0; green, 0; blue, 0 }  ][line width=0.75]    (10.93,-3.29) .. controls (6.95,-1.4) and (3.31,-0.3) .. (0,0) .. controls (3.31,0.3) and (6.95,1.4) .. (10.93,3.29)   ;
\draw   (322,72.5) .. controls (322,66.7) and (326.7,62) .. (332.5,62) .. controls (338.3,62) and (343,66.7) .. (343,72.5) .. controls (343,78.3) and (338.3,83) .. (332.5,83) .. controls (326.7,83) and (322,78.3) .. (322,72.5) -- cycle ;
\draw  [fill={rgb, 255:red, 0; green, 255; blue, 118 }  ,fill opacity=1 ] (322,116.5) .. controls (322,110.7) and (326.7,106) .. (332.5,106) .. controls (338.3,106) and (343,110.7) .. (343,116.5) .. controls (343,122.3) and (338.3,127) .. (332.5,127) .. controls (326.7,127) and (322,122.3) .. (322,116.5) -- cycle ;
\draw   (322,160.5) .. controls (322,154.7) and (326.7,150) .. (332.5,150) .. controls (338.3,150) and (343,154.7) .. (343,160.5) .. controls (343,166.3) and (338.3,171) .. (332.5,171) .. controls (326.7,171) and (322,166.3) .. (322,160.5) -- cycle ;
\draw  [fill={rgb, 255:red, 74; green, 144; blue, 226 }  ,fill opacity=0.85 ] (322,208.5) .. controls (322,202.7) and (326.7,198) .. (332.5,198) .. controls (338.3,198) and (343,202.7) .. (343,208.5) .. controls (343,214.3) and (338.3,219) .. (332.5,219) .. controls (326.7,219) and (322,214.3) .. (322,208.5) -- cycle ;
\draw  [fill={rgb, 255:red, 74; green, 144; blue, 226 }  ,fill opacity=0.85 ] (411,91.5) .. controls (411,85.7) and (415.7,81) .. (421.5,81) .. controls (427.3,81) and (432,85.7) .. (432,91.5) .. controls (432,97.3) and (427.3,102) .. (421.5,102) .. controls (415.7,102) and (411,97.3) .. (411,91.5) -- cycle ;
\draw  [fill={rgb, 255:red, 74; green, 144; blue, 226 }  ,fill opacity=0.85 ] (411,143.5) .. controls (411,137.7) and (415.7,133) .. (421.5,133) .. controls (427.3,133) and (432,137.7) .. (432,143.5) .. controls (432,149.3) and (427.3,154) .. (421.5,154) .. controls (415.7,154) and (411,149.3) .. (411,143.5) -- cycle ;
\draw  [fill={rgb, 255:red, 0; green, 255; blue, 118 }  ,fill opacity=0.85 ] (411,199.5) .. controls (411,193.7) and (415.7,189) .. (421.5,189) .. controls (427.3,189) and (432,193.7) .. (432,199.5) .. controls (432,205.3) and (427.3,210) .. (421.5,210) .. controls (415.7,210) and (411,205.3) .. (411,199.5) -- cycle ;
\draw [fill={rgb, 255:red, 74; green, 144; blue, 226 }  ,fill opacity=0.85 ]   (343,72.5) -- (411,91.5) ;
\draw [fill={rgb, 255:red, 0; green, 255; blue, 118 }  ,fill opacity=1 ]   (343,116.5) -- (411,91.5) ;
\draw [fill={rgb, 255:red, 74; green, 144; blue, 226 }  ,fill opacity=0.85 ]   (343,160.5) -- (411,91.5) ;
\draw [fill={rgb, 255:red, 0; green, 255; blue, 118 }  ,fill opacity=1 ]   (343,116.5) -- (411,143.5) ;
\draw [fill={rgb, 255:red, 0; green, 255; blue, 118 }  ,fill opacity=0.85 ]   (343,160.5) -- (411,199.5) ;
\draw [fill={rgb, 255:red, 0; green, 255; blue, 118 }  ,fill opacity=0.85 ]   (343,208.5) -- (411,199.5) ;
\draw [fill={rgb, 255:red, 74; green, 144; blue, 226 }  ,fill opacity=0.85 ]   (343,208.5) -- (411,143.5) ;

\draw (133,64) node [anchor=north west][inner sep=0.75pt]  [font=\footnotesize] [align=left] {$\displaystyle v_{1}$};
\draw (133,108) node [anchor=north west][inner sep=0.75pt]  [font=\footnotesize] [align=left] {$\displaystyle v_{2}$};
\draw (133,152) node [anchor=north west][inner sep=0.75pt]  [font=\footnotesize] [align=left] {$\displaystyle v_{3}$};
\draw (133,200) node [anchor=north west][inner sep=0.75pt]  [font=\footnotesize] [align=left] {$\displaystyle v_{4}$};
\draw (222,83) node [anchor=north west][inner sep=0.75pt]  [font=\footnotesize] [align=left] {$\displaystyle h_{1}$};
\draw (222,135) node [anchor=north west][inner sep=0.75pt]  [font=\footnotesize] [align=left] {$\displaystyle h_{2}$};
\draw (222,191) node [anchor=north west][inner sep=0.75pt]  [font=\footnotesize] [align=left] {$\displaystyle h_{3}$};
\draw (325,65) node [anchor=north west][inner sep=0.75pt]  [font=\footnotesize] [align=left] {$\displaystyle v_{1}$};
\draw (325,109) node [anchor=north west][inner sep=0.75pt]  [font=\footnotesize] [align=left] {$\displaystyle v_{2}$};
\draw (325,153) node [anchor=north west][inner sep=0.75pt]  [font=\footnotesize] [align=left] {$\displaystyle v_{3}$};
\draw (325,201) node [anchor=north west][inner sep=0.75pt]  [font=\footnotesize] [align=left] {$\displaystyle v_{4}$};
\draw (414,84) node [anchor=north west][inner sep=0.75pt]  [font=\footnotesize] [align=left] {$\displaystyle h_{1}$};
\draw (414,136) node [anchor=north west][inner sep=0.75pt]  [font=\footnotesize] [align=left] {$\displaystyle h_{2}$};
\draw (414,192) node [anchor=north west][inner sep=0.75pt]  [font=\footnotesize] [align=left] {$\displaystyle h_{3}$};

\end{tikzpicture}

    \caption{An improved zero forcing set on the incidence graph of the hypergraph in Figure~\ref{fig:hypergraph-ex}.}
    \label{fig:smaller-zero-forcing-set}
\end{figure}
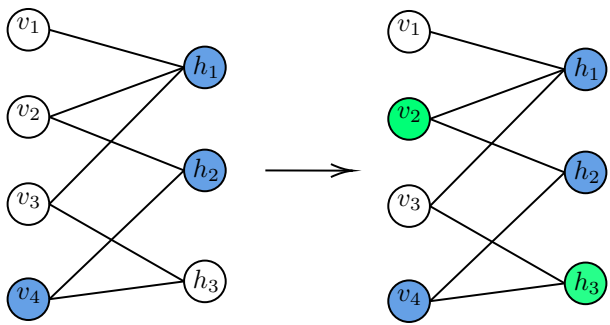

The above observation leads to an improved relation between the lazy burning number on a hypergraph and the zero forcing number of its incidence graph.
For this, the connected components of a hypergraph need to be considered with some observations.
Let $H$ be a hypergraph with $k$ connected components $\{H_i\}_{1\leq i\leq k}$ and let $V(H_i)$ and $E(H_i)$ be the vertex set and hyperedge set of $H_i$, respectively. 
Let $B_i$ be a lazy burning set of $H_i$. 
Since all the connected component are maximal connected sub hypergraphs of $H$, there is no path between the vertices in $V(H_i)$ and $V(H_j)$ for $1\leq i < j \leq k$, we have $B_i\cap B_j = \varnothing$.
Therefore, $b_L(H) = \sum_{i=1}^k b_L(H_i)$.
This means that we can sum the lazy burning numbers of the connected components of a hypergraph to obtain the lazy burning number.
Additionally, by Theorem~\ref{thm:connected components zero forcing number}, the zero forcing number of a hypergraph is equal to the sum of the zero forcing numbers of its connected components.
Considering the above, the bound in Theorem~\ref{thm:hypergraph bound} can be refined as follows.

\begin{theorem}
    Let $H$ be a hypergraph with $k$ connected components that contain at least one non-singleton hyperedge. Then
    \begin{equation}
        Z(IG(H)) \leq b_L(H) + |E(H)| - k.
        \label{eq:component hypergraph inequality}
    \end{equation}
    \label{thm:improved bound}
\end{theorem}
\begin{proof}
    First, we show that $Z(IG(H_i)) \leq b_L(H_i) + |E(H_i)| - 1$ for each of the $k$ connected components $\{H_i\}_{1\leq i\leq k}$ containing at least one non-singleton hyperedge, which is done with proof by contradiction. Then, the relation is shown for the whole hypergraph.
    So, let $\{H_i\}_{1\leq i\leq l}$ be the set of all the $l\geq k$ connected components of $H$ such that $H_i$ for $1\leq i \leq k$ contains at least one non-singleton hyperedge and that $H_i$ for $k< i\leq l$ does contain none or only singleton hyperedges.

\begin{description}
    \item[Case that $B_i\neq \emptyset$.] Take any $v\in B_i$ and take $h\in E(H_i)$ such that $v\in h$ and $h$ is non-singleton. Such $v$ and $h$ exist since $H_i$ is connected and contains at least one non-singleton hyperedge. 
    Earlier, it was shown that $B_i\cup E(H_i)$ was a zero forcing set for $IG(H_i)$. 
    Now, consider the set $B_i\cup E(H_i) \backslash \{h\}$ and we claim that it is a zero forcing set in $IG(H_i)$. 
    Since all the vertices of $IG(H_i)$ that are in $E(H_i)$ apart from $h$ are colored, and $v$ is connected in $IG(H_i)$ to vertices in $E(H_i)$, in the first round the force $v\rightarrow h$ will occur as $h$ is the only uncolored neighbor of $v$. 
    Then at least the set $B_i\cup E(H_i)$ is colored after the first round, which means that the set $B_i\cup E(H_i) \backslash \{h\}$ will color $V(IG(H_i))$.
    With this, it can be concluded that $B_i\cup E(H_i) \backslash \{h\}$ is indeed a zero forcing set for $IG(H_i)$.
    \item[Case $B_i=\emptyset$.] This case is only possible if there is at least one singleton hyperedge $h\in E(H_i)$. Then consider the vertex $v\in h$. Since $H_i$ is connected and there is at least one non-singleton hyperedge, there is a non-singleton hyperedge $h'\in E(H_i)$ such that $v\in h'$. 
    Then consider the set $E(H_i)\backslash\{h'\}$ (as $B_i=\emptyset$). In the first round of the zero forcing process on $IG(H_i)$, the force $h\rightarrow v$ will occur and in the second round the force $v\rightarrow h'$ will occur.
    Therefore, at least $E(H_i)$ will be colored after the second round of the zero forcing process and hence $E(H_i)\backslash \{h'\}$ is a zero forcing set.
    
    From this, the following inequality is obtained for $1\leq i \leq k$:
    
    \[Z(IG(H_i)) \leq b_L(H_i) + |E(H_i)| - 1.\]
    
    Additionally, inequality~\eqref{eq:hypergraph inequality} can be applied to each of the connected component $H_i$ for $k < i \leq l$ of $H$. Then, summing over all the connected components we obtain
    \begin{align*}
        Z(IG(H)) &= \sum_{i=1}^{k} Z(IG(H_i)) + \sum_{i=k+1}^l Z(IG(H_i)) \\
        &\leq \sum_{i=1}^{k} (b_L(H_i)+|E(H_i)|-1) + \sum_{i=k+1}^l (b_L(H_i)+|E(H_i)|) \\
        &= b_L(H) + |E(H)| - k.
 \end{align*} 
\end{description}
This completes the proof.
\end{proof}

The next natural question is when equality is attained for the bound in Theorem~\ref{thm:improved bound}. The following example shows that the above bound can be met with equality.

\begin{example}
    Consider the hypergraph $H$ in Figure~\ref{fig:hypergraph-equality-components} with vertex set $V(H)=\{v_1,\dots, v_7\}$ and hyperedge set $E(H)=\{h_1,h_2\}$ with $h_1=\{v_1,v_2,v_3\}$ and $h_2=\{v_4,v_5,v_6,v_7\}$. One can easily verify that $b_L(H)=5$ and $Z(IG(H))=5$. 
    Note that $|E(H)|=2$ and the number of connected components with at least one hyperedge is $k=2$. 
    Then, the bound in Theorem~\ref{thm:improved bound} holds with equality.
    \label{example:equality-hypergraph}
\end{example}

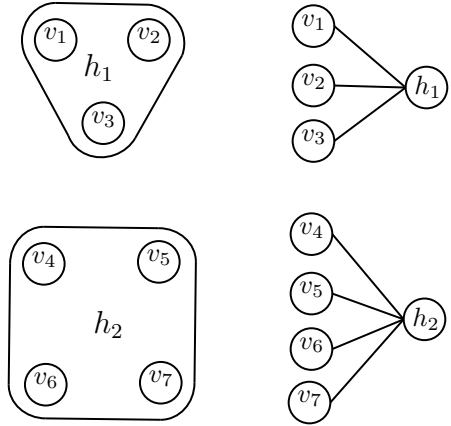
\begin{figure}[htp!]
    \centering

\tikzset{every picture/.style={line width=0.75pt}} 

\begin{tikzpicture}[x=0.75pt,y=0.75pt,yscale=-1,xscale=1]

\draw   (138,178.5) .. controls (138,172.7) and (142.7,168) .. (148.5,168) .. controls (154.3,168) and (159,172.7) .. (159,178.5) .. controls (159,184.3) and (154.3,189) .. (148.5,189) .. controls (142.7,189) and (138,184.3) .. (138,178.5) -- cycle ;
\draw   (196,177.5) .. controls (196,171.7) and (200.7,167) .. (206.5,167) .. controls (212.3,167) and (217,171.7) .. (217,177.5) .. controls (217,183.3) and (212.3,188) .. (206.5,188) .. controls (200.7,188) and (196,183.3) .. (196,177.5) -- cycle ;
\draw   (139,239.5) .. controls (139,233.7) and (143.7,229) .. (149.5,229) .. controls (155.3,229) and (160,233.7) .. (160,239.5) .. controls (160,245.3) and (155.3,250) .. (149.5,250) .. controls (143.7,250) and (139,245.3) .. (139,239.5) -- cycle ;
\draw   (197,238.5) .. controls (197,232.7) and (201.7,228) .. (207.5,228) .. controls (213.3,228) and (218,232.7) .. (218,238.5) .. controls (218,244.3) and (213.3,249) .. (207.5,249) .. controls (201.7,249) and (197,244.3) .. (197,238.5) -- cycle ;
\draw   (144,65.5) .. controls (144,59.7) and (148.7,55) .. (154.5,55) .. controls (160.3,55) and (165,59.7) .. (165,65.5) .. controls (165,71.3) and (160.3,76) .. (154.5,76) .. controls (148.7,76) and (144,71.3) .. (144,65.5) -- cycle ;
\draw   (191,65.5) .. controls (191,59.7) and (195.7,55) .. (201.5,55) .. controls (207.3,55) and (212,59.7) .. (212,65.5) .. controls (212,71.3) and (207.3,76) .. (201.5,76) .. controls (195.7,76) and (191,71.3) .. (191,65.5) -- cycle ;
\draw   (168,107.5) .. controls (168,101.7) and (172.7,97) .. (178.5,97) .. controls (184.3,97) and (189,101.7) .. (189,107.5) .. controls (189,113.3) and (184.3,118) .. (178.5,118) .. controls (172.7,118) and (168,113.3) .. (168,107.5) -- cycle ;
\draw   (274,58.5) .. controls (274,52.7) and (278.7,48) .. (284.5,48) .. controls (290.3,48) and (295,52.7) .. (295,58.5) .. controls (295,64.3) and (290.3,69) .. (284.5,69) .. controls (278.7,69) and (274,64.3) .. (274,58.5) -- cycle ;
\draw   (274,88.5) .. controls (274,82.7) and (278.7,78) .. (284.5,78) .. controls (290.3,78) and (295,82.7) .. (295,88.5) .. controls (295,94.3) and (290.3,99) .. (284.5,99) .. controls (278.7,99) and (274,94.3) .. (274,88.5) -- cycle ;
\draw   (274,116.5) .. controls (274,110.7) and (278.7,106) .. (284.5,106) .. controls (290.3,106) and (295,110.7) .. (295,116.5) .. controls (295,122.3) and (290.3,127) .. (284.5,127) .. controls (278.7,127) and (274,122.3) .. (274,116.5) -- cycle ;
\draw   (331,89.5) .. controls (331,83.7) and (335.7,79) .. (341.5,79) .. controls (347.3,79) and (352,83.7) .. (352,89.5) .. controls (352,95.3) and (347.3,100) .. (341.5,100) .. controls (335.7,100) and (331,95.3) .. (331,89.5) -- cycle ;
\draw    (295,58.5) -- (331,89.5) ;
\draw    (295,88.5) -- (331,89.5) ;
\draw    (295,116.5) -- (331,89.5) ;
\draw   (273,163.5) .. controls (273,157.7) and (277.7,153) .. (283.5,153) .. controls (289.3,153) and (294,157.7) .. (294,163.5) .. controls (294,169.3) and (289.3,174) .. (283.5,174) .. controls (277.7,174) and (273,169.3) .. (273,163.5) -- cycle ;
\draw   (273,193.5) .. controls (273,187.7) and (277.7,183) .. (283.5,183) .. controls (289.3,183) and (294,187.7) .. (294,193.5) .. controls (294,199.3) and (289.3,204) .. (283.5,204) .. controls (277.7,204) and (273,199.3) .. (273,193.5) -- cycle ;
\draw   (273,221.5) .. controls (273,215.7) and (277.7,211) .. (283.5,211) .. controls (289.3,211) and (294,215.7) .. (294,221.5) .. controls (294,227.3) and (289.3,232) .. (283.5,232) .. controls (277.7,232) and (273,227.3) .. (273,221.5) -- cycle ;
\draw   (330,206.5) .. controls (330,200.7) and (334.7,196) .. (340.5,196) .. controls (346.3,196) and (351,200.7) .. (351,206.5) .. controls (351,212.3) and (346.3,217) .. (340.5,217) .. controls (334.7,217) and (330,212.3) .. (330,206.5) -- cycle ;
\draw    (294,163.5) -- (330,206.5) ;
\draw    (294,193.5) -- (330,206.5) ;
\draw    (294,221.5) -- (330,206.5) ;
\draw   (272,248.5) .. controls (272,242.7) and (276.7,238) .. (282.5,238) .. controls (288.3,238) and (293,242.7) .. (293,248.5) .. controls (293,254.3) and (288.3,259) .. (282.5,259) .. controls (276.7,259) and (272,254.3) .. (272,248.5) -- cycle ;
\draw    (293,248.5) -- (330,206.5) ;
\draw  [draw opacity=0] (140.16,75.54) .. controls (138.47,72.64) and (137.5,69.2) .. (137.5,65.5) .. controls (137.5,55.17) and (145.11,46.8) .. (154.5,46.8) -- (154.5,65.5) -- cycle ; \draw   (140.16,75.54) .. controls (138.47,72.64) and (137.5,69.2) .. (137.5,65.5) .. controls (137.5,55.17) and (145.11,46.8) .. (154.5,46.8) ;  
\draw  [draw opacity=0] (201.5,47.71) .. controls (205.79,48.12) and (210.1,50.05) .. (213.52,53.48) .. controls (219.09,59.04) and (220.72,66.93) .. (218.25,73.31) -- (201.5,65.5) -- cycle ; \draw   (201.5,47.71) .. controls (205.79,48.12) and (210.1,50.05) .. (213.52,53.48) .. controls (219.09,59.04) and (220.72,66.93) .. (218.25,73.31) ;  
\draw    (154.5,46.8) -- (201.5,47.71) ;
\draw  [draw opacity=0] (194.62,115.02) .. controls (192.44,118.74) and (188.87,121.82) .. (184.31,123.47) .. controls (174.61,127.01) and (164.14,122.72) .. (160.93,113.9) -- (178.5,107.5) -- cycle ; \draw   (194.62,115.02) .. controls (192.44,118.74) and (188.87,121.82) .. (184.31,123.47) .. controls (174.61,127.01) and (164.14,122.72) .. (160.93,113.9) ;  
\draw    (140.16,75.54) -- (160.93,113.9) ;
\draw    (218.25,73.31) -- (194.62,115.02) ;
\draw  [draw opacity=0] (131.5,178.5) .. controls (131.5,178.5) and (131.5,178.5) .. (131.5,178.5) .. controls (131.5,168.17) and (139.11,159.8) .. (148.5,159.8) -- (148.5,178.5) -- cycle ; \draw   (131.5,178.5) .. controls (131.5,178.5) and (131.5,178.5) .. (131.5,178.5) .. controls (131.5,168.17) and (139.11,159.8) .. (148.5,159.8) ;  
\draw  [draw opacity=0] (206.5,160.5) .. controls (206.5,160.5) and (206.5,160.5) .. (206.5,160.5) .. controls (216.83,160.5) and (225.2,168.11) .. (225.2,177.5) -- (206.5,177.5) -- cycle ; \draw   (206.5,160.5) .. controls (206.5,160.5) and (206.5,160.5) .. (206.5,160.5) .. controls (216.83,160.5) and (225.2,168.11) .. (225.2,177.5) ;  
\draw    (148.5,159.8) -- (206.5,160.5) ;
\draw    (131.5,178.5) -- (130.8,239.5) ;
\draw  [draw opacity=0] (224.5,238.5) .. controls (224.5,238.5) and (224.5,238.5) .. (224.5,238.5) .. controls (224.5,248.83) and (216.89,257.2) .. (207.5,257.2) -- (207.5,238.5) -- cycle ; \draw   (224.5,238.5) .. controls (224.5,238.5) and (224.5,238.5) .. (224.5,238.5) .. controls (224.5,248.83) and (216.89,257.2) .. (207.5,257.2) ;  
\draw  [draw opacity=0] (149.5,256.5) .. controls (149.5,256.5) and (149.5,256.5) .. (149.5,256.5) .. controls (139.17,256.5) and (130.8,248.89) .. (130.8,239.5) -- (149.5,239.5) -- cycle ; \draw   (149.5,256.5) .. controls (149.5,256.5) and (149.5,256.5) .. (149.5,256.5) .. controls (139.17,256.5) and (130.8,248.89) .. (130.8,239.5) ;  
\draw    (225.2,177.5) -- (224.5,238.5) ;
\draw    (149.5,256.5) -- (207.5,257.2) ;

\draw (141,171) node [anchor=north west][inner sep=0.75pt]  [font=\footnotesize] [align=left] {$\displaystyle v_{4}$};
\draw (199,170) node [anchor=north west][inner sep=0.75pt]  [font=\footnotesize] [align=left] {$\displaystyle v_{5}$};
\draw (142,232) node [anchor=north west][inner sep=0.75pt]  [font=\footnotesize] [align=left] {$\displaystyle v_{6}$};
\draw (200,231) node [anchor=north west][inner sep=0.75pt]  [font=\footnotesize] [align=left] {$\displaystyle v_{7}$};
\draw (172,201) node [anchor=north west][inner sep=0.75pt]   [align=left] {$\displaystyle h_{2}$};
\draw (147,58) node [anchor=north west][inner sep=0.75pt]  [font=\footnotesize] [align=left] {$\displaystyle v_{1}$};
\draw (194,58) node [anchor=north west][inner sep=0.75pt]  [font=\footnotesize] [align=left] {$\displaystyle v_{2}$};
\draw (171,100) node [anchor=north west][inner sep=0.75pt]  [font=\footnotesize] [align=left] {$\displaystyle v_{3}$};
\draw (167,71) node [anchor=north west][inner sep=0.75pt]   [align=left] {$\displaystyle h_{1}$};
\draw (277,51) node [anchor=north west][inner sep=0.75pt]  [font=\footnotesize] [align=left] {$\displaystyle v_{1}$};
\draw (277,81) node [anchor=north west][inner sep=0.75pt]  [font=\footnotesize] [align=left] {$\displaystyle v_{2}$};
\draw (277,109) node [anchor=north west][inner sep=0.75pt]  [font=\footnotesize] [align=left] {$\displaystyle v_{3}$};
\draw (334,82) node [anchor=north west][inner sep=0.75pt]  [font=\footnotesize] [align=left] {$\displaystyle h_{1}$};
\draw (276,156) node [anchor=north west][inner sep=0.75pt]  [font=\footnotesize] [align=left] {$\displaystyle v_{4}$};
\draw (276,186) node [anchor=north west][inner sep=0.75pt]  [font=\footnotesize] [align=left] {$\displaystyle v_{5}$};
\draw (276,214) node [anchor=north west][inner sep=0.75pt]  [font=\footnotesize] [align=left] {$\displaystyle v_{6}$};
\draw (333,199) node [anchor=north west][inner sep=0.75pt]  [font=\footnotesize] [align=left] {$\displaystyle h_{2}$};
\draw (275,241) node [anchor=north west][inner sep=0.75pt]  [font=\footnotesize] [align=left] {$\displaystyle v_{7}$};

\end{tikzpicture}

    \caption{A hypergraph and its incidence graph that attain equality in Theorem~\ref{thm:improved bound}.}
    \label{fig:hypergraph-equality-components}
\end{figure}

The following result shows that there are infinitely many hypergraphs for which equality is attained in Theorem~\ref{thm:improved bound}. 

\begin{proposition}{}\label{propo:tightnessnewboundburningzeroforcingprocesses}
    There are infinitely many hypergraphs $H$ for which it holds that
    \[ Z(IG(H)) = b_L(H)+|E(H)|-k, \]
    where $k$ is the number of connected components of $H$ that contain at least one non-singleton hyperedge.
    \label{propo:infinite graphs}
\end{proposition}

\begin{proof}
    Let $H$ be a hypergraph with $l$ connected components and let $\{H_i\}_{i\in[l]}$ be the set of connected components of $H$ with $V(H_i)$ and $E(H_i)$ being the vertex and edge set of $H_i$, such that $H_i$ for $i\in [l]$ contains exactly one non-singleton hyperedge (for an example, see the hypergraph in Example~\ref{example:equality-hypergraph}).
    
    For each of the connected component $H_i$, the number of vertices that it contains is denoted with $n_i=|V(H_i)|$. 
    First, consider the connected components $\{H_{i}\}_{1\leq i \leq k}$. The lazy burning number of such a component is $n_i - 1$ as there is only one hyperedge that contains all the vertices. Additionally, note that $IG(H_i)$ is a star graph with $n_i+1$ vertices. Therefore, it holds that $Z(IG(H_i))=n_i - 1$ (see e.g. \cite{hogben2022inverse}).     
    It can easily be verified that equality is attained in Theorem~\ref{thm:improved bound} for all connected components.

    We then obtain the following:
    \begin{align*}
        Z(IG(H)) &= \sum_{i=1}^l Z(IG(H_i))\\
        &= \sum_{i=1}^l (b_L(H_i)+|E(H_i)|-1)\\
        &= b_L(H) + |E(H)| - l.
    \end{align*}
    Infinitely many such hypergraphs exist. Indeed, consider, for instance, the  family of hypergraphs $\{H_n\}_{n\in \mathbb{N}}$ with $V(H_n)=\{ v_1,v_2,\dots, v_{n+1} \}$ and $E(H_n)=\{\{v_1,v_2,\dots, v_{n+1}\}\}$. This shows the desired result.
\end{proof}
Another hypergraph for which equality in Theorem~\ref{thm:improved bound} is attained and which is not covered by Proposition~\ref{propo:infinite graphs} is shown in the following example.

\begin{example}
    Consider the hypergraph $H$ in Figure~\ref{fig:hypergraph-equality}, which has vertex set $V(H)=\{v_!,v_2,v_3,v_4\}$ and hyperedge set $E(H)=\{h_1,h_2\}$ with $h_1=\{v_2,v_3,v_4\}$ and $h_2=\{v_1,v_2,v_3,v_4\}$.
    One can easily verify that $b_L(H)=2$ and that $Z(IG(H))=3$.
    This implies that we obtain an equality between the two sides in the inequality from Theorem~\ref{thm:improved bound}.
    \label{example:more-than-one-hyperedge-equality}
\end{example}
\vspace{-0.5cm}

\begin{figure}[htp!]
    \centering

\tikzset{every picture/.style={line width=0.75pt}} 

\begin{tikzpicture}[x=0.75pt,y=0.75pt,yscale=-1,xscale=1]

\draw   (125,123.5) .. controls (125,117.7) and (129.7,113) .. (135.5,113) .. controls (141.3,113) and (146,117.7) .. (146,123.5) .. controls (146,129.3) and (141.3,134) .. (135.5,134) .. controls (129.7,134) and (125,129.3) .. (125,123.5) -- cycle ;
\draw   (125,209.5) .. controls (125,203.7) and (129.7,199) .. (135.5,199) .. controls (141.3,199) and (146,203.7) .. (146,209.5) .. controls (146,215.3) and (141.3,220) .. (135.5,220) .. controls (129.7,220) and (125,215.3) .. (125,209.5) -- cycle ;
\draw   (225,209.5) .. controls (225,203.7) and (229.7,199) .. (235.5,199) .. controls (241.3,199) and (246,203.7) .. (246,209.5) .. controls (246,215.3) and (241.3,220) .. (235.5,220) .. controls (229.7,220) and (225,215.3) .. (225,209.5) -- cycle ;
\draw   (225,125.5) .. controls (225,119.7) and (229.7,115) .. (235.5,115) .. controls (241.3,115) and (246,119.7) .. (246,125.5) .. controls (246,131.3) and (241.3,136) .. (235.5,136) .. controls (229.7,136) and (225,131.3) .. (225,125.5) -- cycle ;
\draw  [draw opacity=0] (116.48,123.5) .. controls (116.48,113.67) and (124.99,105.7) .. (135.5,105.7) -- (135.5,123.5) -- cycle ; \draw   (116.48,123.5) .. controls (116.48,113.67) and (124.99,105.7) .. (135.5,105.7) ;  
\draw  [draw opacity=0] (154.52,209.5) .. controls (154.52,219.33) and (146.01,227.3) .. (135.5,227.3) .. controls (124.99,227.3) and (116.48,219.33) .. (116.48,209.5) -- (135.5,209.5) -- cycle ; \draw   (154.52,209.5) .. controls (154.52,219.33) and (146.01,227.3) .. (135.5,227.3) .. controls (124.99,227.3) and (116.48,219.33) .. (116.48,209.5) ;  
\draw    (116.48,123.5) -- (116.48,209.5) ;
\draw  [draw opacity=0] (234.5,106.12) .. controls (239.01,106.42) and (243.53,108.36) .. (247.09,111.91) .. controls (254.52,119.34) and (254.9,131) .. (247.95,137.95) .. controls (244.33,141.58) and (239.42,143.21) .. (234.5,142.88) -- (234.5,124.5) -- cycle ; \draw   (234.5,106.12) .. controls (239.01,106.42) and (243.53,108.36) .. (247.09,111.91) .. controls (254.52,119.34) and (254.9,131) .. (247.95,137.95) .. controls (244.33,141.58) and (239.42,143.21) .. (234.5,142.88) ;  
\draw    (135.5,105.7) -- (234.5,106.12) ;
\draw  [draw opacity=0] (154.48,160.5) .. controls (154.48,150.67) and (162.99,142.7) .. (173.5,142.7) -- (173.5,160.5) -- cycle ; \draw   (154.48,160.5) .. controls (154.48,150.67) and (162.99,142.7) .. (173.5,142.7) ;  
\draw    (154.48,160.5) -- (154.52,209.5) ;
\draw    (173.5,142.7) -- (234.5,142.88) ;
\draw  [draw opacity=0] (112.48,119.5) .. controls (112.48,109.67) and (120.99,101.7) .. (131.5,101.7) -- (131.5,119.5) -- cycle ; \draw   (112.48,119.5) .. controls (112.48,109.67) and (120.99,101.7) .. (131.5,101.7) ;  
\draw    (131.5,101.7) -- (238.5,102.48) ;
\draw  [draw opacity=0] (238.5,102.48) .. controls (248.33,102.48) and (256.3,110.99) .. (256.3,121.5) -- (238.5,121.5) -- cycle ; \draw   (238.5,102.48) .. controls (248.33,102.48) and (256.3,110.99) .. (256.3,121.5) ;  
\draw  [draw opacity=0] (256,213.6) .. controls (256,223.43) and (247.48,231.4) .. (236.98,231.4) -- (236.98,213.6) -- cycle ; \draw   (256,213.6) .. controls (256,223.43) and (247.48,231.4) .. (236.98,231.4) ;  
\draw    (256,213.6) -- (256.3,121.5) ;
\draw    (112.18,211.6) -- (112.48,119.5) ;
\draw  [draw opacity=0] (129.98,230.62) .. controls (120.14,230.62) and (112.18,222.11) .. (112.18,211.6) -- (129.98,211.6) -- cycle ; \draw   (129.98,230.62) .. controls (120.14,230.62) and (112.18,222.11) .. (112.18,211.6) ;  
\draw    (129.98,230.62) -- (236.98,231.4) ;

\draw (128,202) node [anchor=north west][inner sep=0.75pt]  [font=\footnotesize] [align=left] {$\displaystyle v_{3}$};
\draw (229,202) node [anchor=north west][inner sep=0.75pt]  [font=\footnotesize] [align=left] {$\displaystyle v_{4}$};
\draw (128,116) node [anchor=north west][inner sep=0.75pt]  [font=\footnotesize] [align=left] {$\displaystyle v_{1}$};
\draw (229,118) node [anchor=north west][inner sep=0.75pt]  [font=\footnotesize] [align=left] {$\displaystyle v_{2}$};
\draw (176,114) node [anchor=north west][inner sep=0.75pt]   [align=left] {$\displaystyle h_{1}$};
\draw (195,181) node [anchor=north west][inner sep=0.75pt]   [align=left] {$\displaystyle h_{2}$};

\end{tikzpicture}
    \caption{A hypergraph that attains equality in Theorem~\ref{thm:improved bound}.}
    \label{fig:hypergraph-equality}
\end{figure}
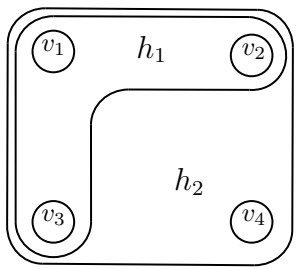

\section{The burning number of a graph does not follow from the spectrum}\label{sec:cospectralburning}

 While it is known that the zero forcing number is not determined by the graph adjacency spectrum \cite{aida-zero-forcing}, this problem remained open for other processes such as graph burning. The goal of this section is to establish this.

Before we can state the main result of this section, we recall two graph products which we will use for the construction of cospectral graph classes with different burning numbers.  The \textit{Cartesian product} of two graphs $G$ and $H$, denoted as $G\square H$, is the graph with vertex set $V(G)\times V(H)$ where two vertices $(u_1,v_1), (u_2,v_2)\in V(G)\times V(H)$ are adjacent if and only if:
    \begin{itemize}
        \item[1.] $u_1=u_2$ and $\{v_1,v_2\}\in E(H)$ or\\ \vspace{-.6cm}
        \item[2.] $v_1=v_2$ and $\{u_1,u_2\}\in E(G)$.
    \end{itemize}

The \textit{strong product} of two graphs $G$ and $H$, denoted as  $G\boxtimes H$, is the graph with vertex set $V(G)\times V(H)$ where two vertices $(u_1,v_1), (u_2,v_2)\in V(G)\times V(H)$ are adjacent if and only if:
    \begin{itemize}
        \item[1.] $\{v_1,v_2\}\in E(H)$ and $u_1=u_2$ or\\ \vspace{-.6cm}
        \item[2.] $\{u_1,u_2\}\in E(G)$ and $v_1=v_2$ or\\ \vspace{-.6cm}
        \item[3.] $\{v_1,v_2\}\in E(H)$ and $\{u_1,u_2\}\in E(G)$.
    \end{itemize}

For graph products, it is sometimes useful to know how the resulting adjacency matrix looks like. It is known that the adjacency matrix of the strong product of two graphs can be obtained from the adjacency matrices of the two original matrices by using the matrix Kronecker product.

    Let $A$ be an $m\times n$-matrix and $B$ be a $p\times q$-matrix, then the \textit{Kronecker product} $A\otimes B$ is the $mp\times nq$-matrix defined by
    \begin{align*}
        A\otimes B = \begin{bmatrix}
            a_{11}b_{11} & \dots & a_{11}b_{1q} & \dots & a_{1n}b_{11}& \dots & a_{1n}b_{1q}\\
            \vdots & \ddots & \vdots & \ddots & \vdots & \ddots & \vdots \\
            a_{11}b_{p1} & \dots & a_{11}b_{pq} & \dots & a_{1n}b_{p1}& \dots & a_{1n}b_{pq}\\
            \vdots & \ddots & \vdots & \ddots & \vdots & \ddots & \vdots \\
            a_{m1}b_{11} & \dots & a_{m1}b_{1q} & \dots & a_{mn}b_{11}& \dots & a_{mn}b_{1q}\\
            \vdots & \ddots & \vdots & \ddots & \vdots & \ddots & \vdots \\
            a_{m1}B_{p1} & \dots & a_{m1}b_{pq} & \dots & a_{mn}b_{p1} & \dots & a_{mn}b_{pq}
        \end{bmatrix}.
    \end{align*}

\begin{proposition}\cite{adjacnecy-matrix-strong-product}
    Let $G=(V(G),E(G))$ and $H=(V(H),E(H))$ be two graphs, and let $A_G$ and $A_H$ be their adjacency matrices, respectively. Then
    \begin{equation*}
        ((A_G+I)\otimes (A_H+I))-I
    \end{equation*}
    is the adjacency matrix of $G\boxtimes H$, where $I$ is the identity matrix.
    \label{prop:adjacency-matrix-strong-product}
\end{proposition}

Let $G$ and $H$ be two graphs and let $\sigma_G$ and $\sigma_H$ be their spectra, respectively. Of interest is the spectrum of the strong product $\sigma_{G\boxtimes H}$ of the graph $G\boxtimes H$.

\begin{lemma}\cite{adjacnecy-matrix-strong-product}
Let $G$ and $H$ be graphs with spectrum $\sigma_G$ and $\sigma_H$ respectively. Then the spectrum of the strong product $G\boxtimes H$ is given by
    \begin{equation*}
        \sigma_{G\boxtimes H}=\{(\lambda+1)(\mu+1)-1 \mid \lambda \in \sigma_G, \mu \in \sigma_H\}.
    \end{equation*}
    \label{lemma:spectrum-strong-product}
\end{lemma}
\vspace{-1.1cm}

The following result is straightforward, but we add its proof for completeness.

\begin{proposition}  Consider the graphs $G_1$, $G_2$ and $H$. If $G_1$ and $G_2$ are cospectral, then $G_1\boxtimes H$ and $G_2\boxtimes H$ are also cospectral.
    \label{prop:remain cospectral strong product}
\end{proposition}

\begin{proof}
    Let $\sigma_{G_1}$, $\sigma_{G_2}$ and $\sigma_H$ denote the spectra of the graphs $G_1$, $G_2$ and $H$ respectively. Since $G_1$ and $G_2$ are cospectral, it holds that $\sigma_{G_1}=\sigma_{G_2}$. 
    Then, using Lemma~\ref{lemma:spectrum-strong-product}, $\sigma_{G_1 \boxtimes H}=\{(\lambda +1)(\mu+1)-1 \mid \lambda\in\sigma_{G_1},\mu\in\sigma_H\} = \{(\lambda +1)(\mu+1)-1 \mid \lambda\in\sigma_{G_2},\mu\in\sigma_H\} = \sigma_{G_2 \boxtimes H}$. 
    This means that $G_1\boxtimes H$ and $G_2\boxtimes H$ have the same spectrum and are hence cospectral.
\end{proof}

Although we saw in Section \ref{sec:burning} that a bound relating the zero forcing number and the lazy burning number exists, in general an inequality connecting the burning numbers of the two original graphs $G$ and $H$ to the burning numbers of the two newly constructed graphs is not sufficient to guarantee that the inequality $b(G)\neq b(H)$ will hold. Hence, one is especially interested in cases for which the burning number of the original graph is equal to the burning number of the new graph constructed using the strong product.
The next result provides such an equality (that is, preservation of the burning number) by considering the strong product of a graph with a complete graph.

\begin{proposition}
    Let $G$ be a graph with burning number $k=b(G)$. Furthermore, assume that there is an optimal burning sequence $(v_1,v_2,\dots,v_k)$ in $G$ such that $v_k$ is a redundant source. 
    Then, $b(G\boxtimes K_n)=b(G)$ for all $n\in \mathbb{N}$ where $K_n$ is the complete graph on $n$ vertices.
    \label{prop:preservation burning number strong product}
\end{proposition}

\begin{proof} 
    If $n=1$, the graphs $G$ and $G\boxtimes K_1$ are isomorphic and therefore have the same burning number. Now let $n\geq 2$.
    Denote the vertex set of $K_n$ as $V(K_n) = \{u_1,u_2,\dots, u_n\}$.
    Let $S_G=(v_1,v_2,\dots,v_k)$ be an optimal burning sequence where $v_k$ is a redundant source.
    First, $b(G\boxtimes K_n)\leq b(G)$ is shown by proving that $S_{G\boxtimes H}=(w_1,w_2,\dots, w_k)$ is a burning sequence in $G\boxtimes K_n$ where $w_i=(v_i,u_1)$. 
    Since $S_G$ is a burning sequence of $G$ and $v_k$ is a redundant source, it holds that $\cup_{j=1}^{k-1}N_{k-j}[v_j]=V(G)$.
    We will now show that $N_j[(v_i,u_{m})] = N_j[v_i]\times V(K_n)$ for all $i\in [k-1]$, $m\in [n]$ and $j\geq 1$. Let $i\in [k-1]$, $m\in [n]$ and $j\geq 1$. 
    \begin{description}
        \item[Case $\subseteq$:] Let $w=(v,u)\in N_j[(v_i,u_m)]$. Take the induced path \linebreak $P=((v_{i_1},u_{m_1}),(v_{i_2},u_{m_2}),\dots,(v_{i_l},u_{m_l}))$
        in $G\boxtimes K_n$ of length $l\leq j$, where \linebreak $(v_{i_1},u_{m_1})=(v,u)$ and $(v_{i_l},u_{m_l})=(v_i,u_m)$. 
        Then, by definition of the strong product, the induced path $P=(v_{i_1},v_{i_2},\dots,v_{i_l})$
        is a path in $G$ between $v$ and $v_i$ of length $l\leq j$ and hence $v\in N[v_i]$.
        \item[Case $\supseteq$:] Let $w=(v,u)\in N_j[v_i]\times V(K_n)$. Then take the induced path \linebreak $P=(v_{i_1},v_{i_2},\dots,v_{i_l})$
        of length $l\leq j$ in $G$ with $v_{i_1}=v$ and $v_{i_l}=v_i$. Then note that the induced path $P=((v_{i_1},u),(v_{i_2},u_m),\dots,(v_{i_l},u_m))$
        is a path in $G\boxtimes K_n$ of length $l\leq j$ and hence $w\in N_j[(v_i,u_m)]$. 
    \end{description}
    With this, the following chain of equalities is obtained:
    \begin{align*}
        \cup_{j=1}^{k-1}N_{k-j}[(v_j,u_{1})] &= \cup_{j=1}^{k-1}\left( N_{k-j}[v_j]\times V(K_n) \right)\\
        &= \left( \cup_{j=1}^{k-1}N_{k-j}[v_j] \right) \times V(K_n) \\
        &= V(G) \times V(K_n) = V(G\boxtimes K_n).
    \end{align*}
    This shows that all the vertices of $G\boxtimes K_n$ get burned by this sequence and that $w_k$ is a redundant source in $S_{G\boxtimes K_n}$. It remains to show that $d(w_i,w_j)\geq |i-j|$ for all $i,j\in [k]$. Suppose that there are $i,j\in [k]$ such that $d(w_i,w_j)=l< |i-j|$. 
    Take an induced path $P=((v_{i_1},u_{m_1}),(v_{i_2},u_{m_2}),\dots,(v_{i_l},u_{m_l}))$
    of length $l < |i-j|$ in $G\boxtimes K_n$ where $(v_{i_1},u_{m_1})=w_i$ and $(v_{i_l},u_{m_l})=w_j$. 
    Then the induced path $P=(v_{i_1},v_{i_2},\dots,v_{i_l})$ is a path in $G$ of length $l < |i-j|$ between $v_i$ and $v_j$ which means that $S_G$ is not a burning sequence of $G$ which is a contradiction.
    According to Theorem~\ref{thm:alternative burning}, $S_{G\boxtimes K_n}$ is a burning sequence of $G\boxtimes K_n$. 
    
    Now, suppose that $b(G\boxtimes K_n) < b(G)$. Then take an optimal burning sequence \linebreak $S_{G\boxtimes K_n} = ((v_1,u_{m_1}),(v_2,u_{m_2}),\dots,(v_l,u_{m_l}))$ of length $l= b(G\boxtimes K_n)$ in $G\boxtimes K_n$. Now note that
    \begin{align*}
        V(G)\times V(K_n) &= V(G\boxtimes K_n)\\
        &= \cup_{j=1}^l N_{l-j}[(v_j,u_{m_j})]\\
        &= \left( \cup_{j=1}^{l-1} N_{l-j}[(v_j,u_{m_j})]\right) \cup (v_l,u_{m_l})\\
        &= \left( \left( \cup_{j=1}^{l-1}N_{l-j}[v_j] \right) \times V(K_n) \right) \cup (v_l,u_{m_l}).
    \end{align*}
    Since $n\geq 2$, it must hold that $\cup_{j=1}^{l-1}N_{l-j}[v_j] = V(G)$ and hence the sequence $S_{G}' = (v_1,v_2,\dots,v_l)$ is a burning sequence of $G$ of length $l<k$, which contradicts the assumption that $S_G$ was an optimal burning sequence of $G$. So we therefore conclude that $b(G\boxtimes K_n)=b(G)$.
\end{proof}

For some graphs, there may not exist an optimal burning sequence that ends in a redundant source. 
Take, for example, a path graph $P_n$ where $n$ is a square number, i.e., $n=m^2$ for some $m\in \mathbb{N}$. 
We know that its burning number is $b(P_n) = m$, so take any optimal burning sequence $S$ of $P_n$ of length $m$. Suppose that the last source is redundant. Then a similar argument can be made as to derive the burning number of a path and at most $m^2 - 1$ vertices are burned.
This means that $S$ cannot be a burning sequence as $P_n$ contains $n=m^2$ vertices.
So $S$ cannot end with a redundant source.
In some graphs, there are optimal burning sequences that do and do not end with a source that is redundant.
The graph in Figure~\ref{fig:graph burning} is such a graph.
The burning sequence $(v_2,v_5,v_8)$ ends with a redundant source and the burning sequence $(v_6,v_1,v_5)$ does not end with a redundant source.
Finally, there are graphs in which all optimal burning sequences end with a redundant source.
Take the complete graphs $K_n$ with $n\geq 2$. In the first round, a source is chosen and since it is a complete graph, all other vertices will get burned in the second round. Hence, any source that is chosen in the second round will be redundant.

At this point, we have all the preliminary results needed to construct an infinite family of cospectral graphs with different burning numbers. Finally, it only remains to show the existence of such cospectral graph pairs.

\begin{example}
    Consider the graphs $G$ and $H$ in Figure~\ref{fig:cospectral graphs}. One can verify that $b(G)=2$ and $b(H)=3$. An optimal burning sequence for $G$ would be $(v_3,v_1)$ and for $H$ would be $(v_3,v_1,v_6)$ - note that the last source of both sequences is redundant.
Additionally, the characteristic polynomial of the adjacency matrix of both graphs is $p(\lambda) = \lambda^6-7\lambda^4-4\lambda^3+7\lambda^2+4\lambda-1$. The spectrum of a graph are all the roots of its characteristic polynomial, $p(\lambda)=0$. Since both graphs have the same characteristic polynomial, their spectrum is the same. Therefore, $G$ and $H$ are cospectral. These graphs are not isomorphic as the non-increasing sequence of the degrees of the vertices of $G$ is $(5,2,2,2,2,1)$ and of $H$ is $(3,3,3,3,1,1)$, which are not the same.
    \label{ex:cospectral graphs}
\end{example}

\begin{figure}[htp!]
    \centering

\tikzset{every picture/.style={line width=0.75pt}} 

\begin{tikzpicture}[x=0.75pt,y=0.75pt,yscale=-1,xscale=1]

\draw   (112,105.5) .. controls (112,99.7) and (116.7,95) .. (122.5,95) .. controls (128.3,95) and (133,99.7) .. (133,105.5) .. controls (133,111.3) and (128.3,116) .. (122.5,116) .. controls (116.7,116) and (112,111.3) .. (112,105.5) -- cycle ;
\draw [fill={rgb, 255:red, 74; green, 144; blue, 226 }  ,fill opacity=1 ]   (133,108.6) -- (159,120.5) ;
\draw   (112,148.5) .. controls (112,142.7) and (116.7,138) .. (122.5,138) .. controls (128.3,138) and (133,142.7) .. (133,148.5) .. controls (133,154.3) and (128.3,159) .. (122.5,159) .. controls (116.7,159) and (112,154.3) .. (112,148.5) -- cycle ;
\draw   (156,127.5) .. controls (156,121.7) and (160.7,117) .. (166.5,117) .. controls (172.3,117) and (177,121.7) .. (177,127.5) .. controls (177,133.3) and (172.3,138) .. (166.5,138) .. controls (160.7,138) and (156,133.3) .. (156,127.5) -- cycle ;
\draw [fill={rgb, 255:red, 74; green, 144; blue, 226 }  ,fill opacity=1 ]   (122.5,116) -- (122.5,138) ;
\draw [fill={rgb, 255:red, 74; green, 144; blue, 226 }  ,fill opacity=1 ]   (133,144.6) -- (157,132.6) ;
\draw   (199,108.5) .. controls (199,102.7) and (203.7,98) .. (209.5,98) .. controls (215.3,98) and (220,102.7) .. (220,108.5) .. controls (220,114.3) and (215.3,119) .. (209.5,119) .. controls (203.7,119) and (199,114.3) .. (199,108.5) -- cycle ;
\draw   (199,151.5) .. controls (199,145.7) and (203.7,141) .. (209.5,141) .. controls (215.3,141) and (220,145.7) .. (220,151.5) .. controls (220,157.3) and (215.3,162) .. (209.5,162) .. controls (203.7,162) and (199,157.3) .. (199,151.5) -- cycle ;
\draw [fill={rgb, 255:red, 74; green, 144; blue, 226 }  ,fill opacity=1 ]   (211.5,118) -- (211.5,140) ;
\draw [fill={rgb, 255:red, 74; green, 144; blue, 226 }  ,fill opacity=1 ]   (176,132.6) -- (200,146.5) ;
\draw [fill={rgb, 255:red, 74; green, 144; blue, 226 }  ,fill opacity=1 ]   (175,122.5) -- (199,110.5) ;
\draw   (156,80.5) .. controls (156,74.7) and (160.7,70) .. (166.5,70) .. controls (172.3,70) and (177,74.7) .. (177,80.5) .. controls (177,86.3) and (172.3,91) .. (166.5,91) .. controls (160.7,91) and (156,86.3) .. (156,80.5) -- cycle ;
\draw [fill={rgb, 255:red, 74; green, 144; blue, 226 }  ,fill opacity=1 ]   (166.5,91) -- (166.5,117) ;
\draw   (359,106.5) .. controls (359,100.7) and (363.7,96) .. (369.5,96) .. controls (375.3,96) and (380,100.7) .. (380,106.5) .. controls (380,112.3) and (375.3,117) .. (369.5,117) .. controls (363.7,117) and (359,112.3) .. (359,106.5) -- cycle ;
\draw   (359,167.5) .. controls (359,161.7) and (363.7,157) .. (369.5,157) .. controls (375.3,157) and (380,161.7) .. (380,167.5) .. controls (380,173.3) and (375.3,178) .. (369.5,178) .. controls (363.7,178) and (359,173.3) .. (359,167.5) -- cycle ;
\draw [fill={rgb, 255:red, 74; green, 144; blue, 226 }  ,fill opacity=1 ]   (380,106.5) -- (434,106.5) ;
\draw   (434,106.5) .. controls (434,100.7) and (438.7,96) .. (444.5,96) .. controls (450.3,96) and (455,100.7) .. (455,106.5) .. controls (455,112.3) and (450.3,117) .. (444.5,117) .. controls (438.7,117) and (434,112.3) .. (434,106.5) -- cycle ;
\draw   (434,167.5) .. controls (434,161.7) and (438.7,157) .. (444.5,157) .. controls (450.3,157) and (455,161.7) .. (455,167.5) .. controls (455,173.3) and (450.3,178) .. (444.5,178) .. controls (438.7,178) and (434,173.3) .. (434,167.5) -- cycle ;
\draw [fill={rgb, 255:red, 74; green, 144; blue, 226 }  ,fill opacity=1 ]   (444.5,117) -- (444.5,157) ;
\draw [fill={rgb, 255:red, 74; green, 144; blue, 226 }  ,fill opacity=1 ]   (369.5,117) -- (369.5,157) ;
\draw [fill={rgb, 255:red, 74; green, 144; blue, 226 }  ,fill opacity=1 ]   (380,167.5) -- (434,167.5) ;
\draw [fill={rgb, 255:red, 74; green, 144; blue, 226 }  ,fill opacity=1 ]   (378,160) -- (437,113) ;
\draw   (315,106.5) .. controls (315,100.7) and (319.7,96) .. (325.5,96) .. controls (331.3,96) and (336,100.7) .. (336,106.5) .. controls (336,112.3) and (331.3,117) .. (325.5,117) .. controls (319.7,117) and (315,112.3) .. (315,106.5) -- cycle ;
\draw [fill={rgb, 255:red, 74; green, 144; blue, 226 }  ,fill opacity=1 ]   (359,106.5) -- (336,106.5) ;
\draw   (478,167.5) .. controls (478,161.7) and (482.7,157) .. (488.5,157) .. controls (494.3,157) and (499,161.7) .. (499,167.5) .. controls (499,173.3) and (494.3,178) .. (488.5,178) .. controls (482.7,178) and (478,173.3) .. (478,167.5) -- cycle ;
\draw [fill={rgb, 255:red, 74; green, 144; blue, 226 }  ,fill opacity=1 ]   (478,167.5) -- (455,167.5) ;

\draw (114,100) node [anchor=north west][inner sep=0.75pt]   [align=left] {$\displaystyle v_{1}$};
\draw (114,143) node [anchor=north west][inner sep=0.75pt]   [align=left] {$\displaystyle v_{2}$};
\draw (158,123) node [anchor=north west][inner sep=0.75pt]   [align=left] {$\displaystyle v_{3}$};
\draw (201,103) node [anchor=north west][inner sep=0.75pt]   [align=left] {$\displaystyle v_{4}$};
\draw (201,145) node [anchor=north west][inner sep=0.75pt]   [align=left] {$\displaystyle v_{6}$};
\draw (158,75) node [anchor=north west][inner sep=0.75pt]   [align=left] {$\displaystyle v_{5}$};
\draw (361,101) node [anchor=north west][inner sep=0.75pt]   [align=left] {$\displaystyle v_{2}$};
\draw (361,163) node [anchor=north west][inner sep=0.75pt]   [align=left] {$\displaystyle v_{4}$};
\draw (436,100) node [anchor=north west][inner sep=0.75pt]   [align=left] {$\displaystyle v_{3}$};
\draw (436,163) node [anchor=north west][inner sep=0.75pt]   [align=left] {$\displaystyle v_{5}$};
\draw (317,101) node [anchor=north west][inner sep=0.75pt]   [align=left] {$\displaystyle v_{1}$};
\draw (480,163) node [anchor=north west][inner sep=0.75pt]   [align=left] {$\displaystyle v_{6}$};

\end{tikzpicture}

    \caption{Cospectral graphs ($G$ on the left and $H$ on the right) with $b(G)=2$ and $b(H)=3$.}
    \label{fig:cospectral graphs}
\end{figure}
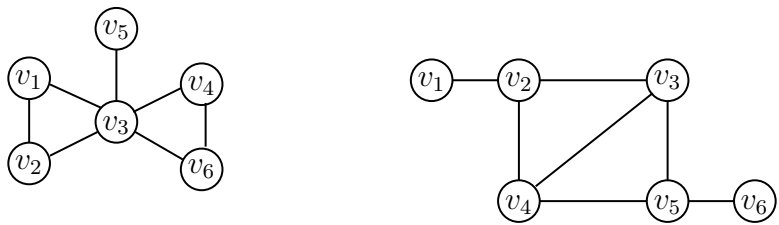

\begin{corollary}\label{coro:burningnumbernotDS}
    Let $G$ and $H$ be two cospectral graphs with respect to the adjacency matrix, such that $b(G)\neq b(H)$. Then there exist infinitely many pairs of cospectral graphs with respect to the adjacency matrix with different burning numbers.
\end{corollary}

\begin{proof}
    Take two cospectral graphs $G$ and $H$ with different burning numbers that have an optimal burning sequence in which the last source is redundant. Example~\ref{ex:cospectral graphs} shows that such graphs exist. 
    Consider the graph products $G\boxtimes K_n$ and $H\boxtimes K_n$ for some $n\geq 2$, which remain cospectral by Proposition~\ref{prop:remain cospectral strong product}. Since both graphs have an optimal burning sequence that ends in a redundant source, it holds that $b(G\boxtimes K_n)=b(G)\neq b(H) = b(H\boxtimes K_n)$ by Proposition~\ref{prop:preservation burning number strong product}. 
    So, the strong product graphs are cospectral and have different burning numbers.
\end{proof}

\subsection*{Acknowledgements}
Aida Abiad is supported by NWO (Dutch Research Council) through the grant \linebreak VI.Vidi.213.085. The authors thank Antonina Khramova for inspiring discussions on this topic and for her feedback on the manuscript.

\printbibliography[heading=bibintoc]

\end{document}